\documentclass[12pt]{amsart}

\usepackage{amsmath,amsthm,amsfonts,amssymb,anysize,bbm,euscript,enumitem,stmaryrd,mathrsfs,xcolor,mathscinet,mathtools,combelow,tikz,tikz-cd,caption,hyperref}
\newtheorem{lemma}{Lemma}[section]
\newtheorem{theorem}[lemma]{Theorem}
\newtheorem{conjecture}[lemma]{Conjecture}
\newtheorem{corollary}[lemma]{Corollary}
\newtheorem{proposition}[lemma]{Proposition}
\newtheorem{example}[lemma]{Example}
\newtheorem{remark}[lemma]{Remark}
\newtheorem{definition}[lemma]{Definition}
\newtheorem{problem}[lemma]{Problem}

\newcommand{\comment}[1]{}
\newcommand{\A}{\mathbb{A}}

\newcommand{\C}{\mathbb{C}}
\newcommand{\N}{\mathbb{N}}
\newcommand{\Z}{\mathbb{Z}}
\newcommand{\Q}{\mathbb{Q}}

\newcommand{\CE}{\mathcal{E}}

\newcommand{\CH}{\mathcal{H}}
\newcommand{\oCH}{\bar{\CE}}
\newcommand{\CK}{\mathcal{K}}
\newcommand{\CL}{\mathcal{L}}

\newcommand{\CS}{\mathcal{S}}

\newcommand{\CV}{\mathcal{V}}

\newcommand{\SBim}{\mathrm{SBim}}
\newcommand{\Hilb}{\mathrm{Hilb}}

\newcommand{\HH}{\mathrm{HH}}
\newcommand{\End}{\mathrm{End}}
\newcommand{\one}{\mathbbm{1}}

\newcommand{\ES}{\EuScript E}

\newcommand{\HS}{\EuScript H}

\newcommand{\PS}{\EuScript P}
\newcommand{\sS}{\EuScript S}

\newcommand{\Hom}{\operatorname{Hom}}

\newcommand{\Br}{\operatorname{Br}}
\newcommand{\Coh}{\operatorname{Coh}}

\newcommand{\Cone}{\operatorname{Cone}}

\newcommand{\Id}{\operatorname{Id}}

\newcommand{\Tr}{\operatorname{Tr}}

\newcommand{\Kar}{\operatorname{Kar}}

\newcommand{\ee}{\mathbf{e}}

\newcommand{\oH}{\bar{E}}
\newcommand{\oHS}{\bar{\ES}}

\newcommand{\comm}{\text{Comm}}
\newcommand{\ocomm}{\overline{\comm}}
\newcommand{\ecomm}{\emph{Comm}}

\newcommand{\fcomm}{\text{FComm}}
\newcommand{\ofcomm}{\overline{\fcomm}}
\newcommand{\efcomm}{\emph{FComm}}

\newcommand{\ecc}{\text{EComm}}
\newcommand{\oecc}{\overline{\ecc}}
\newcommand{\eecc}{\emph{EComm}}

\newcommand{\torus}{\C^* \times \C^*}

\newcommand{\ABr}{\mathrm{ABr}}
\newcommand{\AH}{\mathrm{AH}}
\newcommand{\ASBim}{\mathrm{ASBim}}

\newcommand{\tR}{\widetilde{R}}

\newcommand{\Comm}{\mathrm{Comm}}

\author{Eugene Gorsky}
\address{Department of Mathematics, University of California\\ One Shields Avenue, Davis CA 94702}
\email{egorskiy@math.ucdavis.edu}
\author{Andrei Negu\cb{t}}
\address{École Polytechnique Fédérale de Lausanne (EPFL), Lausanne, Switzerland}
\address{Simion Stoilow Institute of Mathematics (IMAR), Bucharest, Romania}
\email{andrei.negut@gmail.com}

\title{Hecke categories, idempotents, and commuting stacks}

\begin{document}

\begin{abstract} We formulate a connection between a topological and a geometric category. The former is the idempotent completion of the (horizontal) trace of the affine Hecke category, while the latter is the equivariant derived category of the (semi-nilpotent) commuting stack. This provides a more precise and improved version of our proposal in \cite{GN Tr}.

\end{abstract}

\maketitle

\section{Introduction}

\subsection{The problem} 

The main purpose of the present paper is to refine a conjecture of \cite{GN Tr} on a connection between topology (specifically affine Hecke categories) and geometry (specifically commuting stacks). Consider the affine Hecke category $\ASBim_n$ (explicitly defined as the homotopy category of type $\widehat{A}_{n-1}$ Soergel bimodules, see Subsection \ref{sec: affine hecke}), its horizontal trace $\Tr(\ASBim_n)$ and the idempotent completion of the latter
$$
\Tr(\ASBim_n)^{\Kar}
$$
Let us also consider the (semi-nilpotent) commuting stack
\begin{equation}
\label{eqn:comm intro}
\comm_n = \Big \{ X,Y \in \text{Mat}_{n \times n}, Y \text{ nilpotent}, [X,Y] = 0 \Big\} \Big/ GL_n
\end{equation}
with the $GL_n$ action given by simultaneous conjugation of the matrices $X$ and $Y$. There is an action of the torus $\torus$ given by rescaling the matrices $X$ and $Y$ independently.

\begin{problem}
\label{prob:main}

Construct a functor
\begin{equation}
\label{eqn:main functor}
\Tr(\ASBim_n)^{\emph{Kar}} \rightarrow D^b(\Coh_{\torus}(\ecomm_n))
\end{equation}
and identify its essential image. The two gradings by characters of $\torus$ of the category in the right-hand side should correspond to the homological and internal gradings of the category in the left-hand side, see \cite{GNR} for precise formulas on this connection.

\end{problem}

In \cite[Problem 1.1]{GN Tr}, we asked about the existence of a functor akin to  \eqref{eqn:main functor}, but without the word ``Kar" (signifying idempotent, or Karoubi completion) in the left-hand side and by foregoing the condition that $Y$ be nilpotent in \eqref{eqn:comm intro}. This question would be implied by the existence of the functor \eqref{eqn:main functor}, upon composing the latter with the canonical functor
$$
\Tr(\ASBim_n) \rightarrow \Tr(\ASBim_n)^{\Kar}
$$
and with the direct image functor that takes the stack \eqref{eqn:comm intro} to 
\begin{equation}
\label{eqn:stack replace}
\Big \{ X,Y \in \text{Mat}_{n \times n}, [X,Y] = 0 \Big\} \Big/ GL_n
\end{equation}
Thus, we may think of Problem \ref{prob:main} above as a refined version of the main question of \cite{GN Tr}. Even so, the computations that we will do in the present paper are of substantially different nature from the ones of \emph{loc. cit.}, as we will explain in Subsection \ref{sub:comparison}. 

\subsection{Geometry}
\label{sub:intro geom}

In the present paper, we explain how the functor in Problem \ref{prob:main} should behave at the level of objects, and why we believe that it is essentially surjective. To this end, let us note that the Grothendieck group of the right hand side of \eqref{eqn:main functor} was described in \cite{Integral} 
\begin{equation}
\label{eqn:k intro}
K_{\torus}(\comm_n) = \mathop{\bigoplus^{\frac {m_1}{n_1} \leq \dots \leq \frac {m_k}{n_k}}_{(m_i,n_i) \in \Z \times \N}}_{n_1+\dots+n_k = n} \Z[q_1^{\pm 1}, q_2^{\pm 1}] \cdot [\oCH_{m_1,n_1}] \star \dots \star [\oCH_{m_k,n_k}]
\end{equation}
where $\star$ denotes the $K$-theoretic Hall product (\cite{SV2}, see also Subsection \ref{sub:convolution}), $q_1,q_2$ denote the elementary characters of $\C^* \times \C^*$, and $\{[\oCH_{m,n}]\}_{(m,n) \in \Z \times \N}$ are certain $K$-theory classes on the commuting stack that we will recall in Subsection \ref{sub:shuffle formulas 2}. For general $(m,n) \in \Z \times \N$, it is not immediately clear how to lift the $K$-theory class $[\oCH_{m,n}]$ to an object in the derived category. However, formula \cite[equation (4.10)]{Integral} shows how to lift
$$
(1+q_1)(1+q_1+q_1^2) \dots (1+q_1+\dots+q_1^{d-1}) \cdot [\oCH_{m,n}]
$$
to an object in the derived category of $\comm_{n}$, for $d = \gcd(m,n)$. In particular, if $m$ and $n$ are coprime, there is an object (see \eqref{eqn:object h})
$$
\CH_{m,n} \in D^b(\Coh_{\torus}(\comm_n))
$$
whose $K$-theory class is precisely $[\oCH_{m,n}]$ in \eqref{eqn:k intro}. Motivated by parallels with the category $\Tr(\ASBim_n)$ (which we will recall in Subsection \ref{sub:intro top}), we propose the existence of objects $\oCH_{md,nd} \in D^b(\Coh_{\torus}(\comm_{nd}))$ for all $d \geq 1$ and coprime $(m,n) \in \Z \times \N$, as follows. 

\begin{conjecture}
\label{conj:objects intro}

For any coprime $(m,n) \in \Z \times \N$ and $d \geq 1$, there is an action of $S_d$ on
\begin{equation}
\label{eqn:conv intro geom}
\underbrace{\CH_{m,n} \star \dots \star \CH_{m,n}}_{d \text{ times}} \in D^b(\emph{Coh}_{\torus}(\ecomm_{nd}))
\end{equation}
If we let $\oCH_{md,nd}$ denote the $S_d$-antisymmetric part of the object \eqref{eqn:conv intro geom}, then
$$
[\oCH_{md,nd}] = (\text{the generator from \eqref{eqn:k intro}}) \in K_{\torus}(\ecomm_{nd})
$$
and the category $D^b(\emph{Coh}_{\torus}(\ecomm_n))$ is generated by the objects
\begin{equation}
\label{eqn:intro gens geom}
\Big \{ \oCH_{m_1,n_1} \star \dots \star \oCH_{m_k,n_k} \Big \}^{\frac {m_1}{n_1} \leq \dots \leq \frac {m_k}{n_k}}_{\substack{(m_i,n_i) \in \Z \times \N, \\ n_1+\dots+n_k = n}}
\end{equation}
(in the present paper, ``generate" means that any object in the category is isomorphic to a chain complex built out of the objects \eqref{eqn:intro gens geom} and their equivariant shifts). 

\end{conjecture}

While the present paper was being written, Tudor Pădurariu informed us about his ongoing work with Sabin Cautis and Yukinobu Toda (\cite{CPT}), which independently discusses a semiorthogonal decomposition of the category $D^b(\text{Coh}_{\torus}(\comm_n))$ (their construction is more refined than ours, as it contains information on morphisms, and not just objects).

\subsection{Topology}
\label{sub:intro top}

Our own interest in Conjecture \ref{conj:objects intro} was motivated by a parallel statement in topology, which is the main result of the present paper. Recall the extended affine braid group of type $A_{n-1}$
$$
\text{ABr}_n = \Big \langle \omega, \sigma_i \Big \rangle_{i \in \Z/n\Z} \Big/ \Big(\omega \sigma_i \omega^{-1} = \sigma_{i+1}, \sigma_i\sigma_{i+1}\sigma_i = \sigma_{i+1}\sigma_i \sigma_{i+1}, \sigma_i\sigma_j = \sigma_j \sigma_i \Big)_{\forall i,j \in \Z/n\Z, j \neq i\pm 1}
$$
The extended affine Hecke algebra of type $A_{n-1}$ is
$$
\text{AH}_n = \C(q)\left[ \text{ABr}_n \right] \Big/ \Big( (\sigma_i - q)(\sigma_i+q^{-1}) = 0 \Big)_{\forall i \in \Z/n\Z}
$$
In Subsection \ref{sec: affine hecke}, we will recall the monoidal category of affine Soergel bimodules $\ASBim_n$ and its homotopy category $\CK(\ASBim_n)$, which categorifies the algebra $\text{AH}_n$. The main topological object in the present paper is the horizontal trace $\Tr(\ASBim_n)$ (see Subsection \ref{sec: traces} for details), which categorifies $\text{AH}_n / [\text{AH}_n,\text{AH}_n]$. Similarly, there is a natural algebra homomorphism $\text{AH}_n \otimes \text{AH}_{n'} \rightarrow \text{AH}_{n+n'}$ for all $n,n'$, whose categorification is a functor
$$
\CK(\ASBim_n) \otimes \CK(\ASBim_{n'}) \xrightarrow{\star} \CK(\ASBim_{n+n'})
$$
It is widely expected that the functor above descends to the trace categories, and experts are very close to doing so (see for instance \cite{LMRSW}), which would yield a functor
\begin{equation}
\label{eqn:conv intro}
\Tr(\ASBim_n) \otimes \Tr(\ASBim_{n'}) \ ``\xrightarrow{\star}" \ \Tr(\ASBim_{n+n'})
\end{equation}
In the present paper, we only need to invoke the functor \eqref{eqn:conv intro} at the level of objects, in which case its definition is clear. In \cite{GN Tr}, we constructed for all $(m,n) \in \Z \times \N$ an object \footnote{In the notation of \emph{loc. cit.}, we have $H_{m,n} = E_{(d_1(m,n),\dots,d_n(m,n))}$ with $d_i(m,n) = \left \lfloor \frac {mi}n \right \rfloor - \left \lfloor \frac {m(i-1)}n \right \rfloor$.}
\begin{equation}
\label{eqn:old objects intro}
H_{m,n} \in \Tr(\ASBim_n)
\end{equation}

\begin{theorem}
\label{thm:intro}

For any coprime $(m,n) \in \Z \times \N$ and $d \in \N$, there is an action of $S_d$ on
\begin{equation}
\label{eqn:conv intro top}
\underbrace{H_{m,n} \star \dots \star H_{m,n}}_{d \text{ times}} \in \Tr(\ASBim_{nd})
\end{equation}

\end{theorem}

The $S_d$-antisymmetric part of the object \eqref{eqn:conv intro top} will be denoted by
\begin{equation}
\label{eqn:idempotent intro}
\oH_{md,nd} \in \Tr(\ASBim_{nd})^{\text{Kar}}
\end{equation}
and we note that it lies in the idempotent completion of the category $\Tr(\ASBim_{nd})$. In addition to constructing the $S_d$ action in Theorem \ref{thm: endo omega m}, we also describe the full (dg) endomorphism ring of the object \eqref{eqn:conv intro top}, generalizing the results of \cite{GHW,GW}.

\begin{theorem}
\label{thm: intro H via E bar}
For any coprime $(m,n) \in \Z \times \N$ and $d \geq 1$, the objects $H_{md,nd}$ of \eqref{eqn:old objects intro} can be resolved by (meaning that they are homotopy equivalent to complexes built out of) 
$$
\bar{E}_{md_1,nd_1}\star \cdots \star \bar{E}_{md_k,nd_k}
$$
with $d_1+\ldots+d_k=d$. More generally, any product of $H_{md,nd}$ with the same ``slope" 
$$
H_{md_1,nd_1}\star\cdots\star H_{md_k,nd_k}
$$
can be resolved by $\bar{E}_{md'_1,nd'_1}\star \cdots \star \bar{E}_{md'_s,nd'_s}$.
\end{theorem}

We prove Theorem \ref{thm: intro H via E bar} as Corollary \ref{cor: schurs generate in slope 2}. The proof uses the new ``cabling functor'' defined in Section \ref{sec: cabling} and the computations at slope zero from \cite{GW}. This allows us to describe the slope $\frac mn$ subcategories of $\Tr(\ASBim_n)^{\Kar}$ generalizing the slope $\frac mn$ subalgebra in the elliptic Hall algebra (namely $\Lambda^{\frac mn}$ of \eqref{eqn:slope subalgebras}). Furthermore, in Section \ref{sec: ribbon Schur top} we describe an explicit resolution of $H_{md,nd}$ by the products of $\bar{E}_{md_i,nd_i}$, see Theorem \ref{thm: ribbon Schurs slope} and Remark \ref{rem: H schurs slope}.

In \cite{GN Tr}, we proved that the objects 
\begin{equation}
\label{eqn:generate trace intro}
\Big\{ H_{m_1,n_1} \star \dots \star H_{m_k,n_k} \Big\}^{\frac {m_1}{n_1} \leq \dots \leq \frac {m_k}{n_k}}_{\substack{(m_i,n_i) \in \Z \times \N, \\ n_1+\dots+n_k = n}}
\end{equation}
generate $\Tr(\ASBim_n)$ as a triangulated category. 
We would like to state that $\star$ products of the elements \eqref{eqn:idempotent intro} generate the idempotent completions of these categories, but this requires some more refined information about morphisms in $\Tr(\ASBim_n)$ than we currently possess. We summarize our expectations in the following. 

\begin{conjecture}
\label{conj:intro top}
The category $\Tr(\ASBim_n)^{\emph{Kar}}$ is generated by the objects
\begin{equation}
\label{eqn:intro gens top}
\Big\{ \oH_{m_1,n_1} \star \dots \star \oH_{m_k,n_k} \Big\}^{\frac {m_1}{n_1} \leq \dots \leq \frac {m_k}{n_k}}_{\substack{(m_i,n_i) \in \Z \times \N, \\ n_1+\dots+n_k = n}} 
\end{equation}
in the sense that any object in the category is homotopy equivalent to a chain complex built out of the objects \eqref{eqn:intro gens top} and their grading shifts.

\end{conjecture}

Note that by Theorem \ref{thm: intro H via E bar} any object of  $\Tr(\ASBim_n)$ can be indeed resolved by the objects \eqref{eqn:intro gens top}.
Comparing Conjectures \ref{conj:objects intro} and \ref{conj:intro top}, it is natural to conjecture that the sought-for functor in Problem \ref{prob:main} sends
$$
\oH_{m_1,n_1} \star \dots \star \oH_{m_k,n_k} \mapsto \oCH_{m_1,n_1} \star \dots \star \oCH_{m_k,n_k}  
$$
for any $\frac {m_1}{n_1} \leq \dots \leq \frac {m_k}{n_k}$. By the generation statements in Conjectures \ref{conj:objects intro} and \ref{conj:intro top}, this would completely define the functor \eqref{eqn:main functor} on the level of objects.

\subsection{Connection to the Hilbert scheme of points}

Together with Jacob Rasmussen, the authors conjectured in \cite{GNR} the existence of a functor from the category of {\bf finite} Soergel bimodules $\SBim_n$ to the derived category of the (semi-nilpotent) Hilbert scheme of points on $\C^2$. This conjecture was proved (in different frameworks) in \cite{HL} and \cite{OR}. In terms of the present paper, we can rephrase and upgrade the aforementioned conjecture by asking to construct a functor
\begin{equation}
\label{eqn:gnr}
\Tr(\SBim_n)^{\Kar} \rightarrow D^b(\Coh_{\torus}(\Hilb^n(\C^2,\C))).
\end{equation}
In \cite{Elias} Elias conjectured the existence of a ``flattening functor" $\ASBim_n\to \SBim_n$ categorifying the natural homomorphism of Hecke algebras $\AH_n\to \text{H}_n$ (topologically, this functor is expected to  ``flatten" an affine braid diagram to a finite one, hence the name). Despite significant progress in \cite{Tolmachov}, the construction of such a functor in full generality remains elusive. 

Assuming that a suitable flattening functor exists, one could conjecture that it extends to trace categories, and fits into the commutative diagram
$$
\begin{tikzcd}
    \Tr(\ASBim_n)^{\Kar} \arrow{r}{\text{\eqref{eqn:main functor}}} \arrow{d} & D^b(\Coh_{\torus}(\comm_n)) \arrow{d}\\
    \Tr(\SBim_n)^{\Kar} \arrow{r}{\eqref{eqn:gnr}} &  D^b(\Coh_{\torus}(\Hilb^n(\C^2,\C)))
\end{tikzcd}
$$
where the right vertical arrow is pullback with respect to the natural map $\Hilb^n(\C^2,\C) \rightarrow \comm_n$ (which is an open embedding followed by an affine space fibration). 

In more topological terms, the top row of this diagram is expected to categorify the HOMFLY-PT skein of the torus, while  the bottom row categorifies the HOMFLY-PT skein of the solid torus, see \cite[Section 1]{GN Tr} for more context and references. For example, for coprime $(m,n)$ the object $H_{m,n}\in \Tr(\ASBim_n)^{\Kar}$ corresponds to the $(m,n)$ torus knot, and the corresponding sheaf on the Hilbert scheme was (conjecturally) described in \cite{GN ref}. 

\subsection{Relation to other works} 

Motivated by the geometric Langlands program, a number of authors pursued the study of {\bf affine character sheaves}, in particular Ben-Zvi-Nadler-Preygel \cite{BZNP} and Li-Nadler-Yun \cite{LNY}, see also \cite{HHZ}. The category of affine character sheaves is identified with the trace (in the sense of $\infty$-categories) of the affine Hecke category, and with a certain completion of the derived category of the commuting stack.

Instead of working with affine Soergel bimodules, these authors used geometric models for the affine Hecke category developed by Bezrukavnikov \cite{Bez}. One of the key advantages of using Soergel bimodules is that the $q$-grading is readily available, while it is much more delicate in other models. Nevertheless, we expect the two frameworks to be closely related, and it would be extremely interesting to make the connection precise.

As a concrete example, our trace category contains a special object $\Tr(\one)$ and its slope $\frac mn$ analogues $\Tr(\Omega^{md}_{nd})$. In Theorem \ref{thm: endo omega m} we prove that the endomorphism algebras as these objects are quite large and, in particular, contain the affine symmetric group $\widetilde{S}_d$. We expect $\Tr(\one)$ to correspond to the affine Springer sheaf studied in the geometric Langlands literature \cite{BZCHN,BZCHN2,BKV,Zhu}. 

\subsection{Acknowledgments} We would like to thank Tudor Pădurariu for many interesting discussions pertaining to the contents of the present paper, and for sharing with us the progress in the paper \cite{CPT}. We also thank Roman Bezrukavnikov, Pavel Etingof, Ivan Loseu, Xinchun Ma, Francesco Sala and Paul Wedrich for many useful discussions.

The work of E.G. was partially supported by the NSF grant DMS-2302305. The work of A.N. was partially supported by the NSF grant DMS-1845034.

\bigskip

\section{Geometry: derived categories of commuting stacks}

\subsection{The commuting stack} For any $n \in \N$, let us consider the (semi-nilpotent) commuting variety
\begin{equation}
\label{eqn:commuting variety}
\ocomm_n = \Big\{X,Y \in \text{Mat}_{n \times n}, Y \text{ nilpotent}, [X,Y] = 0 \Big\}
\end{equation}
It is endowed with an action of $GL_n$ by simultaneous conjugation of $X$ and $Y$, and with an action of $\torus$ by rescaling $X$ and $Y$ independently. Because these two actions commute, then the $\torus$ action descends to the so-called commuting stack
\begin{equation}
\label{eqn:nilp commuting stack}
\comm_n = \ocomm_n \Big/ GL_n
\end{equation}
The main geometric object considered in the present paper will be the equivariant derived category
\begin{equation}
\label{eqn:derived category}
D^b(\text{Coh}_{\torus}(\comm_n))
\end{equation}

\begin{remark}

As in \cite[Remark 2.1]{GN Tr}, we posit that the correct way to think of \eqref{eqn:commuting variety} is not as an algebraic variety, but as the derived subscheme of affine space cut out by the vanishing of the characteristic polynomial of $Y$ and the quadratic equation $[X,Y] = 0$. Similarly, \eqref{eqn:nilp commuting stack} should be interpreted as a derived stack, so one should work with category \eqref{eqn:derived category} in the context of derived algebraic geometry. We do not stress this aspect in the present paper, as it will not be consequential for us, but we believe that working in the appropriate setting is crucial to solving Problem \ref{prob:main}.

\end{remark}

\subsection{The flag commuting stack} We consider the (semi-nilpotent) flag commuting variety $\ofcomm_n$ parameterizing pairs of commuting $n \times n$ matrices of the form
\begin{equation}
\label{eqn:upper triangular}
X = \begin{pmatrix} x_1 & * & \dots & * \\ 0 & x_2 & \dots & * \\ \vdots & \vdots & \ddots & \vdots \\ 0 & 0 & \dots & x_n \end{pmatrix}, \quad Y = \begin{pmatrix} 0 & * & \dots & * \\ 0 & 0 & \dots & * \\ \vdots & \vdots & \ddots & \vdots \\ 0 & 0 & \dots & 0 \end{pmatrix}
\end{equation}
If $B_n \subset GL_n$ denotes the Borel subgroup of upper triangular matrices, then we may define the flag commuting stack
\begin{equation}
\label{eqn:flag commuting stack}
\fcomm_n = \ofcomm_n \Big / B_n
\end{equation}
which inherits the action of $\torus$ by independently scaling the matrices $X$ and $Y$. In the present paper, we will also encounter the stack
$$
\fcomm_n^\bullet
$$
constructed just as above, but with the additional condition that $x_1=x_2=\dots=x_n$ in \eqref{eqn:upper triangular} (however, the derived structures on the stacks $\fcomm_n^\bullet$ and $\fcomm_n$ are different: while the latter is defined by imposing one equation for each entry of $[X,Y]$ strictly above the diagonal, the former is defined by imposing one equation for each entry of $[X,Y]$ strictly above the superdiagonal, see \cite[Proposition 2.4]{GN Tr} for details). We have natural maps
$$
\pi : \fcomm_n \rightarrow \comm_n
$$
$$
\ \pi^\bullet : \fcomm_n^\bullet \rightarrow \comm_n
$$
induced by the inclusion of pairs of upper triangular matrices into the set of pairs of all matrices. Moreover, we have line bundles
\begin{equation}
\label{eqn:line bundle fcomm}
\begin{tikzcd}
\CL_1, \dots, \CL_n
\arrow[d, dotted] \\
\fcomm_n^\bullet
\end{tikzcd}
\end{equation}
induced by the fact that the stack $\fcomm_n^\bullet$ is a $B_n$ quotient (specifically, the line bundle $\CL_i$ stems from the character $B_n \rightarrow \mathbb{C}^*$ that picks out the $i$-th diagonal entry). If we combine the geometric constructions above, we may define the object
\begin{equation}
\label{eqn:object h}
\CH_{m,n} = \pi^\bullet_* \left( \bigotimes_{i=1}^n \CL_i^{\left \lfloor \frac {mi}n \right \rfloor - \left \lfloor \frac {m(i-1)}n \right \rfloor} \right) \in D^b(\Coh_{\torus}(\comm_n))
\end{equation}
for all $(m,n) \in \Z \times \N$.

\subsection{An eccentric commuting stack}
\label{sub:eccentric}

We will now define a second kind of flag commuting stack, which is inspired (and closely related to) the eccentric correspondences of \cite{Thesis}. We first consider the variety $\oecc_n^\bullet$ parameterizing pairs of commuting $n \times n$ matrices of the following form
\begin{equation}
\label{eqn:eccentric upper triangular}
X = \begin{pmatrix} * & * & * & \dots & * & * & * \\ * & * & * &  \dots & * & * & * \\ 0 & * & * & \dots & * & * & * \\ \vdots & \vdots & \vdots  & \ddots & \vdots & \vdots & \vdots \\ 0 & 0 & 0 & \dots & * & * & * \\ 0 & 0 & 0 & \dots & * & * & *  \\ 0 & 0 & 0 & \dots & 0 & * & * 
 \end{pmatrix}, \quad Y = \begin{pmatrix} 0 & * & * & \dots & * & * & * \\ 0 & 0 & * &  \dots & * & * & * \\ 0 & 0 & 0 & \dots & * & * & * \\ \vdots & \vdots & \vdots  & \ddots & \vdots & \vdots & \vdots \\ 0 & 0 & 0 & \dots & 0 & * & * \\ 0 & 0 & 0 & \dots & 0 & 0 & *  \\ 0 & 0 & 0 & \dots & 0 & 0 & 0 
 \end{pmatrix}
\end{equation}
In other words, the matrices $X$ and $Y$ behave differently with respect to the standard flag
$$
\C^0 \subset \C^1 \subset \dots \subset \C^{n-1} \subset \C^n
$$
in the sense that 
$$
X(\C^i) \subseteq \C^{i+1}, \quad Y(\C^{i+1}) \subseteq \C^i.
$$
The commutator $[X,Y] = 0$ imposes $\frac {n(n+1)}2-1$ equations on the coefficients of $X,Y$, one for every entry on or above the main diagonal (the $-1$ is due to the fact that the sum of the coefficients on the diagonal of $[X,Y]$ yields a superfluous relation due to $\text{Tr}[X,Y] = 0$). As the locus of matrices \eqref{eqn:eccentric upper triangular} is preserved by conjugation with the Borel subgroup $B_n \subset GL_n$ of upper triangular matrices, we may define the eccentric flag commuting stack
\begin{equation}
\label{eqn:eccentric commuting stack}
\ecc^\bullet_n = \oecc^\bullet_n \Big / B_n
\end{equation}
It has an action of $\torus$ by independently scaling the matrices $X$ and $Y$, the map
$$
{\pi'}^\bullet : \ecc_n^\bullet \rightarrow \comm_n
$$
induced by the inclusion of pairs of matrices \eqref{eqn:eccentric upper triangular} into the set of pairs of all matrices, as well as line bundles defined by analogy with \eqref{eqn:line bundle fcomm}
\begin{equation}
\label{eqn:line bundle ecc}
\begin{tikzcd}
\CL'_1, \dots, \CL'_n
\arrow[d, dotted] \\
\ecc_n^\bullet 
\end{tikzcd}
\end{equation}
We may then construct the following object for all $(m,n) \in \Z \times \N$, by analogy with \eqref{eqn:object h}
\begin{equation}
\label{eqn:object h prime}
\CH'_{m,n} = {\pi'_*}^\bullet \left( \bigotimes_{i=1}^n {\CL'_i}^{\left \lceil \frac {mi}n \right \rceil - \left \lceil \frac {m(i-1)}n \right \rceil} \right) \in D^b(\Coh_{\torus}(\comm_n))
\end{equation}

\begin{conjecture}
\label{conj h and h prime}
For any coprime $(m,n) \in \Z \times \N$, we have $\CH_{m,n} \cong \CH'_{m,n} \otimes q_2^{n-1}$.
\end{conjecture}

Our main evidence for Conjecture \ref{conj h and h prime} is the fact that $[\CH_{m,n}] = [\CH'_{m,n}] q_2^{n-1}$ holds in $K$-theory, see \eqref{eqn:h and h prime are equal}. More importantly, the object $\CH'_{m,n}$ was independently studied by other mathematicians (such as Bezrukavnikov-Loseu and Xinchun Ma, to whom we are grateful for discussing their work with us) who also developed versions of Conjecture \ref{conj h and h prime}, see for example \cite[Conjecture 5.3]{Ma}.

\subsection{Convolution} 
\label{sub:convolution}

For any $n,n' \in \N$, let us now recall the convolution operation
\begin{equation}
\label{eqn:convolution}
D^b(\Coh_{\torus}(\comm_{n})) \otimes D^b(\Coh_{\torus}(\comm_{n'})) \xrightarrow{\star} D^b(\Coh_{\torus}(\comm_{n+n'}))
\end{equation}
which is the straightforward generalization of the construction of \cite{SV2} from the level of $K$-theory to derived categories. We refer to \cite{PS} for the difficult technicalities that ensure the above operation is defined and well behaved, and simply sketch the general idea behind the construction here. Consider the parabolic commuting variety $\ocomm_{n,n'}$ parameterizing pairs of commuting $(n+n') \times (n+n')$ matrices $X, Y$ of the form
\begin{equation}
\label{eqn:parabolic matrix}
\left( \begin{array}{c|c}
n \times n \text{ matrix}& * \\
  \hline
0 & n' \times n' \text{ matrix}
\end{array} \right)
\end{equation}
with the diagonal blocks of the $Y$ matrix being nilpotent. Define the stack
\begin{equation}
\label{eqn:parabolic stack}
\comm_{n,n'} = \ocomm_{n,n'} \Big / P_{n,n'}
\end{equation}
where $P_{n,n'} \subset GL_{n+n'}$ is the parabolic subgroup of matrices of shape \eqref{eqn:parabolic matrix}. The maps
\begin{align*}
&p : \comm_{n,n'} \rightarrow \comm_n \\
&p' : \comm_{n,n'} \rightarrow \comm_{n'} \\
&p'' : \comm_{n,n'} \rightarrow \comm_{n+n'}
\end{align*}
remember the top left part, bottom right part and the whole of the matrices \eqref{eqn:parabolic matrix}, respectively. Then the operation \eqref{eqn:convolution} is given by
$$
\star = Rp''_{*} \circ  (Lp^* \boxtimes L{p'}^* )
$$
It is straightforward to check that this operation is associative in the natural sense.

\subsection{An open locus}
\label{sub: slope zero}

Having defined the objects $\CH_{m,n}$ and the convolution $\star$, Conjecture \ref{conj:objects intro} is well-posed, modulo the construction of the $S_d$ action on the object in the left-hand side of \eqref{eqn:conv intro geom}. The full construction will be given in \cite{CPT}, but for illustration we will provide a (rather easy) particular case of this construction when restricted to the open substack
\begin{equation}
\label{eqn:comm circ}
\mathring{\comm}_n \subset \comm_n
\end{equation}
corresponding to the matrix $X$ having distinct eigenvalues (recall that $Y$ is nilpotent, so all its eigenvalues are 0). We define \eqref{eqn:comm circ} as the classical truncation of the open substack in question, so we have an isomorphism
\begin{equation}
\label{eqn:open substack}
\mathring{\comm}_n \cong (\mathbb{C} \times \text{pt} / \mathbb{C}^*)^n/S_n  
\end{equation}
because if $X$ has distinct eigenvalues, then we can use the $GL_n$ action to make $X$ diagonal, which forces the nilpotent matrix $Y$ to be 0. On diagonal matrices, there is a residual trivial action of the maximal torus $T = (\C^*)^n$, and we must take the  quotient by the symmetric group in \eqref{eqn:open substack} because $S_n$ preserves diagonal matrices. Similarly, define
$$
\mathring{\fcomm}_n \subset \fcomm_n
$$
as the open substack corresponding to the matrix $X$ having distinct eigenvalues. We have
\begin{equation}
\label{eqn:iso ring}
\mathring{\fcomm}_n \cong (\mathbb{C} \times \text{pt} / \mathbb{C}^*)^n 
\end{equation}
and the restriction of $\fcomm_n \xrightarrow{\pi} \comm_n$ to the substack $\mathring{\fcomm}_n$ is the quotient map 
$$
\rho : (\mathbb{C} \times \text{pt} / \mathbb{C}^*)^n  \rightarrow (\mathbb{C} \times \text{pt} / \mathbb{C}^*)^n/S_n
$$
By definition, we have
$$
\underbrace{\CH_{0,1} \star \dots \star \CH_{0,1}}_{n \text{ times}} = \pi_*(\mathcal{O}_{\fcomm_n})
$$
and so 
\begin{equation}
\label{eqn:above}
\underbrace{\CH_{0,1} \star \dots \star \CH_{0,1}}_{n \text{ times}} \Big|_{\mathring{\comm_n}} \stackrel{\eqref{eqn:iso ring}}\cong \rho_* \left(\mathcal{O}_{(\mathbb{C} \times \text{pt} /\mathbb{C}^*)^n} \right)
\end{equation}
Explicitly, $\mathcal{O}_{(\mathbb{C} \times \text{pt} /\mathbb{C}^*)^n}$ is the graded ring $\C[x_1,\dots,x_n]$, and the map $\rho_*$ simply remembers it as a module over $\C[x_1,\dots,x_n]^{\text{sym}}$. In particular, the antisymmetric part of this ring is
$$
\C[x_1,\dots,x_n]^{\text{antisym}} = \prod_{1\leq i < j \leq n} (x_i-x_j) \cdot \C[x_1,\dots,x_n]^{\text{sym}}
$$
and thus \eqref{eqn:above} reads
\begin{equation}
\label{eqn:below}
\oCH_{0,n} \Big|_{\mathring{\comm_n}} \stackrel{\eqref{eqn:iso ring}}\cong q_1^{\frac {n(n-1)}2}\mathcal{O}_{\mathring{\comm}_n}
\end{equation}
with $q_1^{\frac {n(n-1)}2}$ being the $\torus$ weight of the Vandermonde determinant $\prod_{1\leq i < j \leq n} (x_i-x_j)$. 

\subsection{$K$-theory}

Consider the Grothendieck groups of the categories \eqref{eqn:derived category}, for all $n \in \mathbb{N}$
$$
K = \bigoplus_{n=0}^{\infty} K_{\torus}(\comm_n)
$$
The natural $K$-theoretic version of the convolution \eqref{eqn:convolution} makes $K$ into an algebra over $\Z[q_1^{\pm 1}, q_2^{\pm 1}]$, where $q_1,q_2$ are the elementary characters of the torus $\torus$. There is an explicit tool for studying $K$ called the shuffle algebra. To define it, consider the map
$$
\iota_n : \ocomm_n \rightarrow \mathbb{A}^{2n^2}
$$
where the right-hand side is the affine space of the coefficients of the matrices $X$ of $Y$ in \eqref{eqn:commuting variety}. The map $\iota_n$ induces a $\Z[q_1^{\pm 1}, q_2^{\pm 1}]$-module homomorphism
\begin{multline*}
K_{\torus}(\comm_n) = K_{\torus \times GL_n}(\ocomm_n) \xrightarrow{\iota_{n*}} \\ \xrightarrow{\iota_{n*}} K_{\torus \times GL_n}(\A^{2n^2}) = \Z[q_1^{\pm 1}, q_2^{\pm 1}][z_1^{\pm 1}, \dots, z_n^{\pm 1}]^{\text{sym}}
\end{multline*}
with the right-most equality due to the fact that $\A^{2n^2}$ is equivariantly contractible (above, $z_1,\dots,z_n$ denote the elementary characters of a maximal torus of $GL_n$, and ``sym" denotes symmetric polynomials). Putting the above maps as $n$ runs from 0 to $\infty$ together gives us a $\Z[q_1^{\pm 1}, q_2^{\pm 1}]$-module homomorphism
$$
K \xrightarrow{\iota_*} \CV = \bigoplus_{n=0}^{\infty} \Z[q_1^{\pm 1}, q_2^{\pm 1}][z_1^{\pm 1}, \dots, z_n^{\pm 1}]^{\text{sym}}
$$
It is well-known that $\iota_*$ is injective (see \cite[Lemma 2.5.1]{VV}), that its image 
$$
\CS = \iota_*(K)
$$
is determined by the explicit conditions in \cite[Definition 3.2]{Integral}, and that $\iota_*$  intertwines the convolution $\star$ on $K$ with the following shuffle product on $\CV$
\begin{multline}
\label{eqn:shuffle}
R(z_1,\dots,z_n) * R'(z_1,\dots,z_{n'}) = \\ = \text{Sym} \left[\frac {R(z_1,\dots,z_n)R'(z_{n+1},\dots,z_{n+n'})}{n!n'!} \prod_{1\leq i \leq n < j \leq n+n'} \zeta \left( \frac {z_i}{z_j} \right) \right]
\end{multline}
where
\begin{equation}
\label{eqn:zeta}
\zeta(x) = \frac {(1-xq_1)(1-xq_2)(1-x^{-1}q_1q_2)}{1-x}
\end{equation}

\subsection{Shuffle algebra formulas 1}
\label{sub:shuffle formulas 1}

From the definition of the object $\CH_{m,n}$ in \eqref{eqn:object h}, it is an exercise (proved along the lines of Proposition \ref{prop:computation} below) to show that its $K$-theory class satisfies
\begin{equation}
\label{eqn:know}
\iota_*[\CH_{m,n}] = \HS_{m,n}
\end{equation}
where
\begin{equation}
\label{eqn:h shuffle}
\HS_{m,n} = (1-q_1)^{n-1} (1-q_2)^n \cdot \text{Sym} \left[\frac {\prod_{i=1}^n z_i^{\left \lfloor \frac {mi}n \right \rfloor - \left \lfloor \frac {m(i-1)}n \right \rfloor}}{\prod_{i=1}^{n-1}\left(1-\frac {z_{i+1}q_1q_2}{z_i}\right)} \prod_{1\leq i < j \leq n} \zeta \left(\frac {z_i}{z_j}\right) \right]
\end{equation}
for all $(m,n) \in \Z \times \N$. Similarly, we have the following.

\begin{proposition}
\label{prop:computation}

The $K$-theory class of the object $\CH'_{m,n}$ in \eqref{eqn:object h prime} satisfies
\begin{equation}
\label{eqn:know prime}
\iota_*[\CH'_{m,n}] = \HS'_{m,n}
\end{equation}
where
\begin{equation}
\label{eqn:h prime shuffle}
\HS'_{m,n} = (1-q_1q_2)^{n-1} (1-q_2)^n \cdot \emph{Sym} \left[\frac {\prod_{i=1}^n z_i^{\left \lceil \frac {mi}n \right \rceil - \left \lceil \frac {m(i-1)}n \right \rceil}}{\prod_{i=1}^{n-1}\left(1-\frac {z_iq_1}{z_{i+1}}\right)} \prod_{1\leq i < j \leq n} \zeta \left(\frac {z_i}{z_j}\right) \right]
\end{equation}
for any $(m,n) \in \Z \times \N$.

\end{proposition}

\begin{proof} Let $\ecc_n \xrightarrow{\pi'} \comm_n$ be the map induced by the inclusion of pairs of matrices of the form \eqref{eqn:eccentric upper triangular} into the set of all pairs of matrices. We need to prove that the composition 
\begin{equation}
\label{eqn:composition 1}
K_{\torus}(\ecc_n) \xrightarrow{\pi'_*} K_{\torus}(\comm_n) \xrightarrow{\iota_{n*}} K_{\torus \times GL_n}(\A^{2n^2})
\end{equation}
sends the class of the line bundle $\bigotimes_{i=1}^n {\CL'_i}^{\left \lceil \frac {mi}n \right \rceil - \left \lceil \frac {m(i-1)}n \right \rceil}$ to the right-hand side of \eqref{eqn:h prime shuffle}. However, in order to better understand this composition, we choose to rewrite $\ecc_n$ as a $GL_n$ quotient, as follows. Consider
$$
\widetilde{\ecc}_n \stackrel{f}\hookrightarrow \A^{2n^2} \times GL_n / B_n \stackrel{g}\twoheadrightarrow \A^{2n^2}
$$
where we think of $\A^{2n^2} \times GL_n / B_n$ as parameterizing pairs of $n \times n$ matrices $X,Y$ together with a full flag
\begin{equation}
\label{eqn:full flag}
0 = V_0 \subset V_1 \subset \dots \subset V_{n-1} \subset V_n =  \C^n
\end{equation}
(and the map $g$ forgets the flag), while $f$ denotes the derived simultaneous zero locus of the equations
\begin{equation}
\label{eqn:equations}
[X,Y] = 0 \quad \text{and} \quad X(V_i) \subseteq V_{i+1} \quad \text{and} \quad Y(V_{i+1})\subseteq V_i
\end{equation}
for all $i \in \{0,\dots,n-1\}$. Then we have
\begin{equation}
\label{eqn:ecc alt}
\ecc_n = \widetilde{\ecc}_n \Big/ GL_n
\end{equation}
which implies that $K_{\torus}(\ecc_n) = K_{\torus \times GL_n}(\widetilde{\ecc}_n)$. We may therefore rewrite the composition \eqref{eqn:composition 1} as the composition
\begin{equation}
\label{eqn:composition 2}
K_{\torus \times GL_n}(\widetilde{\ecc}_n) \xrightarrow{f_*} K_{\torus \times GL_n}(\A^{2n^2} \times GL_n / B_n) \xrightarrow{g_*} K_{\torus \times GL_n}(\A^{2n^2})
\end{equation}
Because the map $f_*$ is the direct image of a derived zero locus, it is given by multiplication with the Koszul complex of the coordinates of the equations \eqref{eqn:equations}. Thus, we have
\begin{align*}
f_*\left(\bigotimes_{i=1}^n {\CL'_i}^{\left \lceil \frac {mi}n \right \rceil - \left \lceil \frac {m(i-1)}n \right \rceil}\right) =& \prod_{i=1}^n L_i^{\left \lceil \frac {mi}n \right \rceil - \left \lceil \frac {m(i-1)}n \right \rceil} \times \\
&(1-q_1q_2)^{n-1} \prod_{1 \leq i < j \leq n} \left(1-\frac {L_jq_1q_2}{L_i} \right)\times \qquad  (\text{from }[X,Y]=0) \\
&\prod_{1 \leq i < j-1 \leq n-1} \left(1-\frac {L_iq_1}{L_j} \right)\times \qquad \qquad \quad \quad  (\text{from }X(V_{i})\subseteq V_{i+1}) \\
&(1-q_2)^n\prod_{1 \leq i < j \leq n} \left(1-\frac {L_iq_2}{L_j} \right) \qquad \qquad (\text{from }Y(V_{i+1})\subseteq V_i)
\end{align*}
where $L_1, \dots,L_n$ denote the $K$-theory classes of the tautological line bundles on $GL_n/B_n$ that parameterize the one-dimensional quotients $V_1/V_0, \dots, V_n/V_{n-1}$. Together with the well-known formula
$$
g_*(R(L_1,\dots, L_n)) = \text{Sym} \left[\frac {R(z_1,\dots,z_n)}{\prod_{1\leq i < j \leq n} \left(1-\frac {z_i}{z_j} \right)} \right]
$$
for integration over the flag variety $GL_n/B_n$ (which holds for any Laurent polynomial $R$), we note that $g_* \circ f_* \left(\otimes_{i=1}^n {\CL'_i}^{\left \lceil \frac {mi}n \right \rceil - \left \lceil \frac {m(i-1)}n \right \rceil}\right)$ precisely equals the right-hand side of \eqref{eqn:h prime shuffle}.

\end{proof}

As shown in \cite{Integral}, when $\gcd(m,n) = 1$ we have
\begin{equation}
\label{eqn:h and h prime are equal}
\HS_{m,n} = \HS'_{m,n} q_2^{n-1}
\end{equation}
which provides evidence for Conjecture \ref{conj h and h prime}.

\subsection{Shuffle algebra formulas 2}
\label{sub:shuffle formulas 2}

We continue to assume that $m \in \Z$ and $n \in \N$ are coprime. We posit that the $K$-theory class of the object $\oCH_{md,nd}$ sought for in Conjecture \ref{conj:objects intro} is given by 
\begin{equation}
\label{eqn:believe}
\iota_*[\oCH_{md,nd}] = \oHS_{md,nd}
\end{equation}
where for all $d \geq 1$ we set
\begin{multline}
\label{eqn:bar h shuffle}
\oHS_{md,nd} = (1-q_1)^{nd} (1-q_2)^{nd} \cdot \\ \text{Sym} \left[\frac {\prod_{i=1}^{nd} z_i^{\left \lfloor \frac {mi}n \right \rfloor - \left \lfloor \frac {m(i-1)}n \right \rfloor}\prod_{i=1}^{d-1} \left(q_1^i - \frac {z_{ni+1}q_1q_2}{z_{ni}} \right)}{\prod_{i=1}^d (1-q_1^i) \prod_{i=1}^{nd-1}\left(1-\frac {z_{i+1}q_1q_2}{z_i}\right)} \prod_{1\leq i < j \leq nd} \zeta \left(\frac {z_i}{z_j}\right) \right]
\end{multline}
Although the right-hand side of the above formula is not manifestly the $K$-theory class of an object in the derived category of $\comm_n$, the main result of \cite{Integral} proves that this is indeed the case. Another reason we believe that \eqref{eqn:believe} holds is the following fact (proved in \cite[Lemma 1.6]{Integral})
\begin{equation}
\label{eqn:iota k}
\CS_n = \mathop{\bigoplus^{\frac {m_1}{n_1} \leq \dots \leq \frac {m_k}{n_k}}_{(m_i,n_i) \in \Z \times \N}}_{n_1+\dots+n_k = n} \Z[q_1^{\pm 1}, q_2^{\pm 1}] \cdot \oHS_{m_1,n_1} * \dots * \oHS_{m_k,n_k}
\end{equation}
where $\CS_n = \iota_*\left( K_{\torus}(\comm_n) \right)$. 

\begin{remark}
\label{rem: n! easy}

By analogy with \eqref{eqn:know}, note that the expression
$$
(1+q_1)(1+q_1+q_1^2) \dots (1+q_1+\dots+q_1^{d-1})\oHS_{md,nd}
$$
(for all coprime $m,n$ and all $d \geq 1$) equals a linear combination of $K$-theory classes of line bundles pushed forward along the map $\pi^\bullet : \efcomm_n^\bullet \rightarrow \ecomm_n$. Similarly, \cite{Integral} shows that
\begin{multline}
\label{eqn:bar h prime shuffle}
\oHS_{md,nd} = (1-q_1^{-1}q_2^{-1})^{nd} (1-q_2)^{nd} \cdot \\ \emph{Sym} \left[\frac {\prod_{i=1}^{nd} z_i^{\left \lfloor \frac {mi}n \right \rfloor - \left \lfloor \frac {m(i-1)}n \right \rfloor}\prod_{i=1}^{d-1} \left(q_1^{-i} q_2^{-i} - \frac {z_{ni+1}}{z_{ni}q_1} \right)}{\prod_{i=1}^d (1-q_1^{-i} q_2^{-i}) \prod_{i=1}^{nd-1}\left(1-\frac {z_{i+1}}{z_i q_1}\right)} \prod_{1\leq i < j \leq nd} \zeta \left(\frac {z_i}{z_j}\right) \right]
\end{multline}
which (by analogy with formula \eqref{eqn:know prime}) implies that 
$$
(1+q_1q_2) (1+q_1q_2 + q_1^2q_2^2) \dots (1+q_1q_2+\dots+q_1^{d-1}q_2^{d-1})\oHS_{md,nd}
$$
equals a linear combination of $K$-theory classes of line bundles pushed forward along the map ${\pi'}^\bullet : \eecc_n^\bullet \rightarrow \ecomm_n$.

\end{remark}

\subsection{A substack}
\label{sub: substack}

Moreover, we have the following identity in $\CV$ (see \cite[Remark 4.9]{Integral})
\begin{equation}
\label{eqn:f to h}
\oHS_{0,n} = q_1^{\frac {n(n-1)}2} \prod_{1 \leq i, j \leq n} \left(1- \frac {z_iq_2}{z_j}\right)
\end{equation}
The expression in the right-hand side of \eqref{eqn:f to h} has long been recognized to play an important role in the study of the algebra $\CS$ (see for instance \cite{FHHSY}). Moreover, the product of factors in the right-hand side of \eqref{eqn:f to h} admits a geometric interpretation, via the closed substack
\begin{equation}
\label{eqn:closed substack}
\text{Mat}_{n \times n} \Big/ GL_n \stackrel{\nu}\hookrightarrow \comm_n, \quad X \mapsto (X,0)
\end{equation}

\begin{proposition}
\label{prop:closed substack}

The structure sheaf of the substack \eqref{eqn:closed substack} satisfies
\begin{equation}
\label{eqn:iota closed substack}
\iota_{n*} \left[\mathcal{O}_{\emph{Mat}_{n \times n} / GL_n}\right] = \prod_{1 \leq i, j \leq n} \left(1- \frac {z_iq_2}{z_j}\right)
\end{equation}

\end{proposition}

\begin{proof} The left-hand side of \eqref{eqn:iota closed substack} is equal to $\iota_{n*} ( \nu_*(1))$. However, the map $\iota_{n} \circ \nu$ is simply equal to the closed embedding of the locus of matrices $(X,0)$ in the locus of matrices $(X,Y)$. This embedding is cut out by explicit equations (the coefficients of $Y$) whose Koszul complex (as $GL_n$ characters) is precisely the right-hand side of \eqref{eqn:iota closed substack}. This establishes Proposition \ref{prop:closed substack}.
\end{proof}

With Proposition \ref{prop:closed substack} in mind, formula \eqref{eqn:f to h} suggests that the correct extension of formula \eqref{eqn:below} to the entire stack $\comm_n$ should be
\begin{equation}
\label{eqn:wishful}
\oCH_{0,n} = q_1^{\frac {n(n-1)}2}\mathcal{O}_{\text{Mat}_{n \times n} / GL_n}
\end{equation}
We hope that \eqref{eqn:wishful} holds for the objects $\oCH_{0,n}$ constructed in \cite{CPT}.

\subsection{Slope subalgebras}

Summing \eqref{eqn:iota k} over all $n \in \N$ gives a PBW-type decomposition
$$
\CS = \bigoplus^{\frac {m_1}{n_1} \leq \dots \leq \frac {m_k}{n_k}}_{(m_i,n_i) \in \Z \times \N} \Z[q_1^{\pm 1}, q_2^{\pm 1}] \cdot \oHS_{m_1,n_1} * \dots * \oHS_{m_k,n_k}
$$
Since the product is taken in non-decreasing order of $\frac {m_i}{n_i}$, then $\CS$ is the ordered product of the so-called slope subalgebras
\begin{equation}
\label{eqn:slope subalgebras}
\Lambda^{\frac mn} = \Z[q_1^{\pm 1}, q_2^{\pm 1}] \Big[\oHS_{m,n}, \oHS_{2m,2n},  \oHS_{3m,3n},\dots \Big]
\end{equation}
for all coprime $(m,n) \in \Z \times \N$. It is well-known that $\oHS_{m,n}, \oHS_{2m,2n}, \oHS_{3m,3n},\dots$ all commute (see \cite{Shuf}), and so we have an isomorphism
$$
\phi^{\frac mn} : \Lambda \xrightarrow{\sim} \Lambda^{\frac mn}
$$
\begin{equation}
\label{eqn:elementary}
\phi^{\frac mn}(\bar{e}_d) = \oHS_{md,nd} 
\end{equation}
where $\Lambda$ denotes the ring of elementary symmetric polynomials in countably many variables (with two parameters $q_1,q_2$, which are none other that $q,t^{-1}$ in Macdonald's notation) and $\bar{e}_d$ denotes the $d$-th elementary symmetric function. We will also have the power sum functions $\bar{p}_d$, determined by the generating series identity
$$
1 + \sum_{d=1}^{\infty} \frac {\bar{e}_d}{(-x)^d} = \exp \left(- \sum_{d=1}^{\infty} \frac {\bar{p}_d}{dx^d} \right)
$$
As shown in \cite{Integral}, we have
\begin{equation}
\label{eqn:power}
\phi^{\frac mn} (\bar{p}_d) = \bar{\PS}_{md,nd}
\end{equation}
where $\bar{\PS}_{md,nd}$ is given by the formula
$$
\frac {(1-q_1)^{nd} (1-q_2)^{nd}}{1-q_1^d} \cdot \text{Sym} \left[\frac {\prod_{i=1}^{nd} z_i^{\left \lfloor \frac {mi}n \right \rfloor - \left \lfloor \frac {m(i-1)}n \right \rfloor}\sum_{i=0}^{d-1} \frac {z_{n(d-1)+1}\dots z_{n(d-i)+1}}{z_{n(d-1)}\dots z_{n(d-i)}} (q_1q_2)^i}{\prod_{i=1}^{nd-1}\left(1-\frac {z_{i+1}q_1q_2}{z_i}\right)} \prod_{1\leq i < j \leq nd} \zeta \left(\frac {z_i}{z_j}\right) \right]
$$
The ring $\Lambda$ is rich in automorphisms, as one can rescale the generators $\bar{p}_d$ independently and arbitrarily. Such automorphisms of $\Lambda$ are called plethysms, and the most important one for our purposes is 
\begin{equation}
\label{eqn:plethysm}
p_d = \bar{p}_d (1-q_1^d)
\end{equation}
If we consider the elements $h_d \in \Lambda$ defined by the formula
$$
1 + \sum_{d=1}^{\infty} \frac {h_d}{x^d} = \exp \left(\sum_{d=1}^{\infty} \frac {p_d}{dx^d} \right)
$$
then it was shown in \cite{Integral} that 
\begin{equation}
\label{eqn:elementary prime}
\phi^{\frac mn}(h_d) = (1-q_1)\HS_{md,nd}
\end{equation}
for all coprime $(m,n) \in \Z \times \N$ and all $d \in \N$. Therefore, if we think of $\bar{e}_d$ (resp. $\oHS_{md,nd}$ for fixed coprime $m,n$) as elementary symmetric functions, then $h_d$ (resp. $(1-q_1)\HS_{md,nd}$ for fixed coprime $m,n$) are plethystically modified complete symmetric functions.

\subsection{Ribbon Schur functions}
\label{sub:ribbon}

For any pair of Young diagrams $\mu \subset \lambda$, one may define a skew Schur function $s_{\lambda \backslash \mu} \in \Lambda$. In the present paper, we will be concerned with the particular case when $\lambda \backslash \mu$ is a ``ribbon", i.e. a connected collection of $1 \times 1$ lattice squares which does not contain any $2 \times 2$ lattice squares. Up to translation, such a ribbon is completely determined by a sequence 
$$
\varepsilon = (\varepsilon_1,\dots,\varepsilon_{d-1}) \in \{\pm\}^{d-1}
$$
and it is defined as the collection of $d$ lattice squares where the $(i+1)$-th square is directly below (respectively to the right of) the $i$-th square if $\varepsilon_i = -$ (respectively $\varepsilon_i = +$). We will write $s_{\varepsilon}$ for the corresponding skew Schur function, and note that the $s_{\varepsilon}$ are completely determined by the properties
\begin{equation}
\label{eqn:product formula}
s_{\varepsilon} s_{\varepsilon'} = s_{\varepsilon,+,\varepsilon'} + s_{\varepsilon,-,\varepsilon'}, \quad \forall \varepsilon,\varepsilon'
\end{equation}
(where $\varepsilon, x,\varepsilon'$ means the concatenation of $\varepsilon, \{x\}$ and $\varepsilon'$) and 
\begin{equation}
\label{eqn:h to ribbon}
h_d = s_{(+,\dots,+)_{d-1 \text{ pluses}}},
\end{equation}
for all $d \in \N$. Then we observe the 
formula
\begin{equation}
\label{eqn:e to ribbon}
\bar{e}_d = \sum_{\varepsilon = (\varepsilon_1, \dots, \varepsilon_{d-1}) \in \{\pm\}^{d-1}} \frac {q_1^{\sum_{1\leq i \leq d-1}^{\varepsilon_i = +} i} \cdot s_{\varepsilon}}{(1-q_1)(1-q_1^2)\dots(1-q_1^d)}
\end{equation}
It was shown in \cite{Integral} that 
$$
\phi^{\frac mn}(s_{\varepsilon}) = \sS^{\frac mn}_{\varepsilon}
$$
where for any coprime $m,n \in \Z \times \N$ and $\varepsilon = (\varepsilon_1,\dots,\varepsilon_{d-1}) \in \{\pm\}^{d-1}$, we set
\begin{multline}
\label{eqn:ribbon shuffle}
\sS^{\frac mn}_{\varepsilon} = (1-q_1)^{nd} (1-q_2)^{nd} \cdot \\ \text{Sym} \left[\frac {\prod_{i=1}^{nd} z_i^{\left \lfloor \frac {mi}n \right \rfloor - \left \lfloor \frac {m(i-1)}n \right \rfloor}\prod_{1 \leq i \leq d-1}^{\varepsilon_i = -} \left(- \frac {z_{ni+1}q_1q_2}{z_{ni}} \right)}{\prod_{i=1}^{nd-1}\left(1-\frac {z_{i+1}q_1q_2}{z_i}\right)} \prod_{1\leq i < j \leq nd} \zeta \left(\frac {z_i}{z_j}\right) \right]
\end{multline}
The formula above clearly categorifies, in the sense that it is equal (up to multiplication by $\pm (1-q_1)$) to the $K$-theory class of the push-forward of an obvious line bundle under the map $\pi^\bullet : \fcomm_{nd}^{\bullet} \rightarrow \comm_{nd}$ (compare with formulas \eqref{eqn:know}, \eqref{eqn:h shuffle}).

\subsection{Old vs new generators}
\label{sub:comparison}

In \cite{GN Tr}, we worked with the following elements of $K \stackrel{\iota_*}{\cong} \CS$
\begin{equation}
\label{eqn:r shuffle}
R_{\boldsymbol{d}} = (1-q_1)^{n-1} (1-q_2)^n \cdot \text{Sym} \left[\frac {z_1^{d_1} \dots z_n^{d_n}}{\prod_{i=1}^{n-1}\left(1-\frac {z_{i+1}q_1q_2}{z_i}\right)} \prod_{1\leq i < j \leq n} \zeta \left(\frac {z_i}{z_j}\right) \right] \in \CS
\end{equation}
defined for all $\boldsymbol{d} = (d_1,\dots,d_n) \in \Z^n$. For instance, the elements $\HS_{m,n}$ of \eqref{eqn:h shuffle} are of the form above for suitably chosen $\boldsymbol{d}$. Meanwhile, the elements $\oHS_{md,nd}$ of \eqref{eqn:bar h shuffle} are linear combinations of objects $R_{\boldsymbol{d}}$ divided by certain elements of $\Z[q_1^{\pm 1},q_2^{\pm 1}]$. We have
$$
\Z[q_1^{\pm 1}, q_2^{\pm 1}]\text{-span of }\{R_{\boldsymbol{d}} \}_{\boldsymbol{d} \in \Z^n} \subsetneq \CS
$$
but 
$$
\Q(q_1,q_2)\text{-span of }\{R_{\boldsymbol{d}} \}_{\boldsymbol{d} \in \Z^n} = \CS \bigotimes_{\Z[q_1^{\pm 1},q_2^{\pm 1}]} \Q(q_1,q_2)
$$
In other words, the elements \eqref{eqn:r shuffle} generate a proper $\Z[q_1^{\pm 1}, q_2^{\pm 1}]$-submodule of $\CS = \iota_*(K)$. By analogy with \eqref{eqn:object h}, the natural categorification of the element $R_{\boldsymbol{d}}$ is the object 
\begin{equation}
\label{eqn:r object}
 \pi^\bullet_* \left( \CL_1^{d_1} \dots \CL_n^{d_n} \right) \in D^b(\Coh_{\torus}(\comm_n))
\end{equation}
The above objects were our main interest in \cite{GN Tr}, and we indicated corresponding objects in the trace of the affine Hecke category under the functor of Problem \ref{prob:main}. However, the objects \eqref{eqn:r object} have no chance of generating the category $D^b(\Coh_{\torus}(\comm_n))$, since they do not even generate its $K$-theory. This is our main reason for proposing that one study the elements $\oHS_{md,nd}$ of \eqref{eqn:bar h shuffle} instead (especially since \cite{CPT} will construct categorifications $\oCH_{md,nd}$ of these elements and show that they generate $D^b(\Coh_{\torus}(\comm_n))$). 

\bigskip

\section{Topology: traces of affine Hecke categories}

\subsection{Notations for affine braids}

We will follow the notation of \cite{Elias,GN Tr} for affine braids. Recall that the (extended) affine braid group $\ABr_n$ is defined by generators $\sigma_0,\ldots,\sigma_{n-1}, \omega$ and relations 
$$
\sigma_i\sigma_{i+1}\sigma_i=\sigma_{i+1}\sigma_i\sigma_{i+1},\ \sigma_i\sigma_j=\sigma_j\sigma_i\ (i \neq j \pm 1),\ \omega\sigma_i\omega^{-1}=\sigma_{i+1},
$$
where the indices are understood modulo $n$. The (extended) affine symmetric group $\widetilde{S_n}$ is a quotient of $\ABr_n$ by the additional relations $\sigma_i^2=1$. The elements $\sigma_1,\ldots,\sigma_{n-1}$ generate a subgroup of $\widetilde{S_n}$ isomorphic to the symmetric group $S_n$ (resp. a subgroup of $\ABr_n$ isomorphic to the braid group $\Br_n$). We will need the  braids
$$
y_i:=\sigma_{i-1}^{-1}\cdots \sigma_1^{-1}\omega \sigma_{n-1}\cdots \sigma_{i}, \qquad \forall i \in \{1,\dots,n\}
$$
which are known to pairwise commute.

There is a grading on both $\ABr_n$ or $\widetilde{S_n}$ defined by
\begin{equation}
\label{eq: degree}
\deg(\omega)=1,\ \deg(\sigma_i)=0.
\end{equation}
Note that $\deg(y_i)=1$. Any element of $\ABr_n$ or $\widetilde{S_n}$ can be uniquely written as $\omega^k\alpha$ where $k\in \Z$ and $\deg(\alpha)=0$. 

The elements $\alpha$ of $\widetilde{S_n}$ of degree zero form a Coxeter group (of type $\widehat{A_{n-1}}$) generated by $\sigma_0,\ldots,\sigma_{n-1}$, and as such have the notions of length $\ell(\alpha)$ and Bruhat order. These extend to the full $\widetilde{S_n}$ by writing 
$$
\ell(\omega^k\alpha)=\ell(\alpha)
$$
and 
$$
\omega^k\alpha\preceq \omega^m\beta \Leftrightarrow \Big( k=m\ \mathrm{and}\ \alpha\preceq \beta \Big).
$$
Given an affine permutation $\alpha\in \widetilde{S_n}$, we can consider its positive braid lift to $\ABr_n$ by considering an arbitrary minimal length representative of $\alpha$ and replacing the generators by their namesakes in $\ABr_n$. Note that $y_i$  are not positive braid lifts of any permutations, but certain monomials in $y_i$ are (see \cite[Lemma 3.2]{GN Tr}). 

We will also use the product 
$$
\star : \ABr_n \times \ABr_{n'} \to  \ABr_{n+n'}
$$ which sends the generators
$\sigma_i,y_i$ of $\ABr_n$ to the namesake generators of $\ABr_{n+n'}$ and the generators $\sigma_j,y_j$ of $\ABr_{n'}$ to
$\sigma_{j+n}, y_{j+n}$ respectively. Topologically, we wrap an affine braid in $\ABr_{n'}$ around an affine braid in $\ABr_n$, see \cite{GN Tr} for more details and pictures. Note that the corresponding product of finite braids
$$
\star : \Br_n \times \Br_{n'} \to  \Br_{n+n'}
$$
is simply the disjoint union (or ``horizontal stacking") of braids since there is no wrapping involved.

\subsection{The affine Hecke category}
\label{sec: affine hecke}

We now define the category $\ASBim_n$ of (extended) affine Soergel bimodules, abbreviated as the affine Hecke category, following \cite{Elias,MT}.

Let $R=\C[x_1,\ldots,x_n]$ and $\tR=\C[x_1,\ldots,x_n,\delta]$, and note that the symmetric group $S_n$ acts on both $R$ and $\tR$ by permuting $x_i$ and fixing $\delta$. The rings $R$ and $\tR$ are graded such that the variables $x_i$ and $\delta$ have grading 2.
Apart from $S_n$, we have an additional endomorphism of $\tR$ given by: 
$$
\omega(\delta)=\delta,\ \omega(x_n)=x_1-\delta,\ \omega(x_i)=x_{i+1},\ 1\le i\le n-1.
$$
It is easy to see that $\omega$ and $S_n$ define an action of $\widetilde{S_n}$ on $\tR$. We will consider $R-R$ (respectively $\tR-\tR$) bimodules, the simplest being $\one = R$ (respectively $\tR$) with the usual left and right multiplication. We will also encounter the  Bott-Samelson bimodules:
$$
\overline{B_i}=R\otimes_{R^{s_i}}R,\qquad  B_i=\tR\otimes_{\tR^{s_i}}\tR 
$$
There is an additional $(\tR,\tR)$-bimodule $\Omega$ which is isomorphic to $\tR$, where the left action of $\tR$ is standard and the right action is twisted by $\omega$. One can check that  $\Omega B_i \Omega^{-1}\simeq B_{i+1}$, where for any $\tR-\tR$ bimodules $M,N$, we consider the $\tR-\tR$ bimodule
$$
MN = M \otimes_{\tR} N 
$$
The category of finite Soergel bimodules $\SBim_n$ is defined as the smallest full subcategory of $R-R$ bimodules containing $R$ and $\overline{B_i}$ and closed under tensor products, direct sums and direct summands. Similarly, the category of (extended) affine Soergel bimodules $\ASBim_n$ is defined as a smallest full subcategory of $\tR-\tR$ bimodules containing $\tR,B_i$ and $\Omega$ and closed under tensor products, direct sums and direct summands.

Note that the subcategory of  $\ASBim_n$ generated by $\tR$ and $B_i$ is equivalent to  $\SBim_n$ with scalars extended from $R$ to $\tR$. As in  \cite{Rouquier}, one can define Rouquier complexes: 
\begin{equation}
\label{eqn:rouquier complexes}
T_i=q^{-1}[B_i\xrightarrow{b_i} \one], \qquad T_i^{-1}=q^{-1}[q^2\one\xrightarrow{b^*_i} B_i]
\end{equation}
which satisfy the braid relations up to homotopy ($q$ records the grading shift). Here the maps between $B_i$ and $\one$ are given by
$$
b_i(1)=1,\quad b^*_i(1)=x_i\otimes 1-1\otimes x_{i+1}.
$$
It is easy to see that $\Omega T_i\Omega^{-1}=T_{i+1}$, so the assignment $\sigma_i \mapsto T_i, \omega \mapsto \Omega$ induces a homomorphism from $\ABr_n$ to the homotopy category $\CK(\ASBim_n)$ (by ``homomorphism" we mean an assignment of an object of $\CK(\ASBim_n)$ for any element of the affine braid group, which intertwines tensor product of objects with the multiplication of affine braids).

Given an affine permutation $v\in \widetilde{S_n}$, we can consider its positive braid lift and the corresponding Rouquier complex $T_v$. Clearly, $T_{\omega^k\alpha}=\Omega^k T_{\alpha}$. Also, we can define Rouquier complexes
$$
Y_i=T_{i-1}^{-1}\cdots T_1^{-1}\Omega T_{n-1}\cdots T_{i}
$$
corresponding to the braids $y_i$.
Computing $\Hom$'s between the products of $Y_i$ and more general Rouquier complexes remains a major open problem, see \cite{Maltoni} for sample computations. Here we will need two partial results.
 
First, Elias proved in \cite{Elias} that whenever $\deg(\alpha)=\deg(\beta)=0$ we have
$$
\Hom_{\ASBim_n}(\Omega^k T_{\alpha},\Omega^{k'}T_{\beta})=\begin{cases}
\Hom_{\ASBim_n}(T_{\alpha},T_{\beta}) & \text{if}\ k=k'\\
0 & \text{otherwise}.
\end{cases}
$$
In particular, $\Hom_{\ASBim_n}(X,Y)$ vanishes unless $\deg(X)=\deg(Y).$
Furthermore, the work of Libedinsky-Williamson \cite{LW} implies the following result. 
\begin{theorem} 
\label{thm:LW}
Let $u,v$ be two permutations in  $\widetilde{S_n}$. Then $\Hom(T_u,T_v)=0$ unless $u\preceq v$ in the Bruhat order. 
\end{theorem}

We expect the existence of dg bifunctors:
$$ 
\CK(\ASBim_n)\boxtimes \CK(\ASBim_{n'})
\xrightarrow{\star} \CK(\ASBim_{n+n'})
$$
whose action on objects matches the product $\ABr_n\times \ABr_{n'} \xrightarrow{\star} \ABr_{n+n'}$  above. If $T_{\alpha}\in\CK(\ASBim_n)$ and $T_{\beta}\in \CK(\ASBim_{n'})$ are Rouquier complexes for affine
braids, then the object $T_{\alpha}\star T_{\beta}=T_{\alpha\star \beta}$ is well defined in $\CK(\ASBim_{n+n'})$ (this follows from the braid
relations). However, defining $\star$ on morphisms remains an open problem. See also \cite{LMRSW} for related higher categorical constructions.

\subsection{The trace categories}
\label{sec: traces}

We will use the formalism of categorical traces, developed in \cite{BHLZ,BPW} and further extended to dg categories in \cite{GHW}. Let $\mathcal{C}$ be a monoidal (dg) category. We define another dg category $\Tr(\mathcal{C})$, called the {\em derived horizontal trace} of $\mathcal{C}$. as following:

\begin{itemize}
\item The objects of $\Tr(\mathcal{C})$ are in bijection with the objects of $\mathcal{C}$. For $X\in \mathcal{C}$ there is an object $\Tr(X)\in \Tr(\mathcal{C})$. 
\item The morphism set $\Hom_{\Tr(\mathcal{C})}(\Tr(X),\Tr(Y))$ is defined as the following complex. For each $k\ge 1$, one considers all possible collections of objects $Z_1,\ldots,Z_k$ in $C$ and all possible sequences of  morphisms
\begin{equation}
\label{eq: Hom dg trace}
\Hom_{\mathcal{C}}(XZ_k,Z_1Y)\otimes \Hom_{\mathcal{C}} (Z_1,Z_2)\otimes \cdots \otimes \Hom _{\mathcal{C}}(Z_{k-1},Z_k).
\end{equation}
The differential, defined in \cite{GHW}, is a variant of the bar complex differential together with the internal differentials on these $\Hom$ complexes.
\item In particular, one can consider 
$$
\Hom^0_{\Tr(\mathcal{C})}(\Tr(X),\Tr(Y))=H^0\left(\Hom_{\Tr(\mathcal{C})}(\Tr(X),\Tr(Y))\right)=\bigoplus_{Z}\Hom_{\mathcal{C}}(XZ,ZY)/\sim
$$
which corresponds to the ``underived" horizontal trace in \cite{BHLZ,BPW}.
\item The composition of morphisms is defined as a variant of the shuffle product on the bar complex, see \cite{GHW} for details.
\item There is a canonical trace functor $\Tr: \mathcal{C}\to \Tr(\mathcal{C})$ which sends $X$ to $\Tr(X)$.
\end{itemize}
We will be interested in the trace categories for $\mathcal{C}=\SBim_n,\ASBim_n$, and implicitly identify these with $\Tr(\CK(\SBim_n)),\Tr(\CK(\ASBim_n))$ after pre-triangulated completion \cite{GHW}.

For $\Tr(\SBim_n)$, we recall some of the main results of \cite{GHW,GW}. First, we describe the endomorphism algebra of the trace of the identity object (also known as the {\em vertical trace}) in the category of Soergel bimodules.

\begin{theorem}(\cite[Proposition 6.20, Theorem 7.6]{GHW})
\label{thm: vertical trace}
We have\footnote{Here and in Theorem \ref{thm: endo omega m} we mean the homology of the endomorphism complex.}
$$
\End_{\Tr(\SBim_n)}(\Tr(\one_n))\simeq \HH_*(R)\rtimes \C[S_n].
$$
In particular,
$
\End^0_{\Tr(\SBim_n)}(\Tr(\one_n))\simeq R\rtimes \C[S_n].
$
\end{theorem}

As a corollary, there is an action of $S_n$ on $\Tr(\one_n)$ by endomorphisms. For each partition $\lambda\vdash n$, we denote by $\ee_{\lambda}\in \C[S_n]$ the projector to the irreducible representation labeled by $\lambda$.

\begin{definition}
We define the object 
$
\Tr(\one_n)^{\lambda}:=\ee_{\lambda}\Tr(\one_n)\in \Tr(\SBim_n)^{\Kar}
$
as the direct summand of $\Tr(\one)$ corresponding to $\ee_{\lambda}$. We will denote 
$
\Tr(\one_n)^{(n)}=\Tr(\one_n)^{\mathrm{sym}}
$
and 
\begin{equation}
\label{eq: def En bar}
\bar{E}_n:=\Tr(\one_n)^{(1^n)}=\Tr(\one_n)^{\mathrm{antisym}}.
\end{equation}
\end{definition}

We would like to emphasize that $\Tr(\one_n)^{\lambda}$ are not objects in the trace category $\Tr(\SBim_n)$, but rather in its idempotent completion $\Tr(\SBim_n)^{\Kar}$. 

Given a graded pre-triangulated dg category $\mathcal{C}$, recall that its (graded, split) Grothendieck group $K_0(\mathcal{C})$ is generated by the isomorphism classes of of objects $[X]$ modulo relations
$$
[X\oplus Y]=[X]+[Y],\ [qX]=q[X],\ [X[1]]=-[X],\ \left[\Cone(X[1]\rightarrow Y)\right]=[Y]-[X].
$$
Here, as above, $qX$ denotes the grading shift and $X[1]$ the homological shift.

\begin{theorem}[\cite{GW}]
\label{thm: schurs generate finite}
a) The objects $\Tr(\one_n)^{\lambda}$ generate the idempotent completion $\Tr(\SBim_n)^{\Kar}$. In other words, any object of $\Tr(\SBim_n)^{\Kar}$ can be resolved by a complex built from $\Tr(\one_n)^{\lambda}$.

b) The (graded) Grothendieck group of $\Tr(\SBim_n)^{\Kar}$ is isomorphic to the space of degree $n$ symmetric functions with coefficients in $\C[q^{\pm 1}]$. The isomorphism identifies the object $\Tr(\one_n)^{\lambda}$ with the Schur function $s_{\lambda}.$
\end{theorem}

In Section \ref{sec: ribbon Schur top} we present some explicit computations of such resolutions. Theorem \ref{thm: schurs generate finite} can be compared with the constructions in Sections \ref{sub: slope zero} and \ref{sub: substack}  as follows. The object $\CH_{0,1}\in D^b(\text{Coh}_{\torus}(\Comm_1))$ corresponds to $H_{0,1}=\Tr(\one_1)\in \Tr(\SBim_1)\subset \Tr(\ASBim_1).$ Similarly, $(\CH_{0,1})^{\star n}$ corresponds to 
\begin{equation}
\label{eq: power H zero one}
(H_{0,1})^{\star n}=(\Tr(\one_1))^{\star n}=\Tr(\one_n)\in \Tr(\SBim_n)\subset \Tr(\ASBim_n).
\end{equation}
We expect that the actions of $S_n$ on $(\CH_{0,1})^{\star n}$ and on $\Tr(\one_n)$  agree, and the antisymmetric components $\oCH_{0,n}$ and $\bar{E}_{n}$ match as well.

\begin{corollary}
The objects 
$\bar{E}_{n_1}\star \cdots \star \bar{E}_{n_k}$
where $n_1+\ldots+n_k=n$
generate the idempotent completion $\Tr(\SBim_n)^{\Kar}$.
\end{corollary}

\begin{proof}
We have
$$
\Tr(\one_{n_1})\star \cdots \star \Tr(\one_{n_k})=\Tr(\one_{n}),
$$
and the isomorphism is compatible with the action of $S_{n_1}\times \cdots\times S_{n_k}\subset S_n$. Therefore 
$\bar{E}_{n_1}\star \cdots \star \bar{E}_{n_k}
$
corresponds to the symmetric function $e_{n_1}\cdots e_{n_k}$, and any Schur function of degree $n$ can be expressed in terms of these. See \cite{GW} for more details.

\end{proof}

Note that in the notation of \cite{GW}, the product 
$\bar{E}_{n_1}\star \cdots \star \bar{E}_{n_k}$ corresponds to $k$ concentric circles colored by $n_1,\ldots,n_k$.

\begin{example}
Let $w_0$ denote the longest element in $S_n$, and let $B_{w_0}$ be the corresponding indecomposable Soergel bimodule. Then by \cite[Example 8.10]{GHW} one has
$$
\Tr(B_{w_0})=[n!]_q\bar{E}_n=(1+q)\cdots (1+q+\ldots+q^{n-1})\bar{E}_n.
$$
In particular, $[n!]_q\bar{E}_n$ belongs to the trace category before Karoubi completion, 
in agreement with Remark \ref{rem: n! easy}.
\end{example}

We also recall some basic properties of $\Tr(\ASBim_n)$ from \cite{GN Tr}.
First, it is clear that $\Hom_{\Tr(\ASBim_n)}(\Tr(X),\Tr(Y))=0$ unless $\deg(X)=\deg(Y)$.   Indeed, for \eqref{eq: Hom dg trace} to be non-zero, one must have $\deg(Z_1)=\ldots=\deg(Z_k)$ and 
$\deg(XZ_k)=\deg(Z_1Y)$, which forces $\deg(X)=\deg(Y)$. Furthermore, we have the following.

\begin{lemma}\cite[Lemma 3.19]{GN Tr}
\label{lem: endo trace rouquier}
For any affine braid $v\in \ABr_n$ we have a natural action of the quotient: 
$$
R_v:=\widetilde{R}\Big/(\delta,x_1-v(x_1), \dots, x_n - v(x_n))
$$
on the trace $\Tr(T_v)$. Here, as above, $v$ acts on $\widetilde{R}$ via its projection to $\widetilde{S_n}$. In particular, the action of $\delta$ on $\Tr(T_v)$ is trivial for all $v$.

\end{lemma}

\subsection{Schur objects with arbitrary slope}

In this section we would like to generalize the above results to an arbitrary slope $\frac mn \in \Q$, as in \eqref{eqn:slope subalgebras}. Following \cite{GN Tr}, we define some natural objects in the trace category. Fix $(m,n)\in \Z\times \N$ such that $\gcd(m,n)=1$. 

\begin{definition}
For any $d\ge 1$ and coprime $(m,n)$ as above we define
$$
H_{md,nd}=\Tr
\left(\prod_{i=1}^{nd}Y_i^{\lfloor\frac{mi}{n}\rfloor-\lfloor\frac{m(i-1)}{n}\rfloor}\cdot T_1\cdots T_{nd-1}\right)\in \Tr(\ASBim_{nd}).
$$
\end{definition}

Note that for $m=0$ and $n=1$ we get
$$
H_{0,d}=\Tr(T_1\cdots T_{d-1}).
$$

\begin{theorem}\cite[Theorems 4.7-4.8]{GN Tr}
\label{thm: convex paths}
The trace category $\Tr(\ASBim_n)$ is generated by the   products 
\begin{equation}
\label{eq: product Pmn}
H_{m_1d_1,n_1d_1}\star H_{m_2d_2,n_2d_2}\cdots \star H_{m_kd_k,n_kd_k}
\end{equation}
where $n_1d_1+\ldots+n_kd_k=n$ and $\frac{m_1}{n_1}\le \frac{m_2}{n_2}\le\cdots \le \frac{m_k}{n_k}$.
\end{theorem}

In what follows we will work with the affine Hecke categories $\ASBim_n$ for different $n$, so we will denote the objects $\Omega, \one$ by $\Omega_n,\one_n \in \ASBim_n$, to avoid confusion.

\begin{lemma}
Suppose that $\gcd(m,n)=1$. Then we have the following isomorphisms in $\Tr(\ASBim_{nd})$:
\begin{equation}
\label{eq: omega from Pmn}
\Tr\left(\Omega^{md}_{nd}\right)\simeq H_{m,n}\star \cdots \star H_{m,n}=\left(H_{m,n}\right)^{\star d}
\end{equation}
and
\begin{equation}
\label{eq: pmn from omega}
\Tr\left(\Omega^{md}_{nd} T_1\cdots T_{d-1}\right)\simeq H_{md,nd}.
\end{equation}
\end{lemma}

\begin{proof}
Both equations follow from \cite[Theorem 4.8]{GN Tr} which states that if $T_v$ is the Rouquier complex for a positive braid lift of a minimal length representative in the conjugacy class of $v\in \widetilde{S_{nd}}$ then $\Tr(T_v)$ can be uniquely written as a product \eqref{eq: product Pmn}. Moreover, $n_i$ correspond to the lengths of cycles in the projection of $v$ to $S_{nd}$, and $m_i$ correspond to the degrees of these cycles, see  \cite[Lemma 3.6]{GN Tr}.

Clearly, $\Omega^{md}_{nd}$ and $\Omega^{md}_{nd} T_1\cdots T_{d-1}$ are Rouquier complexes for positive braid lifts of permutations. In the first case, the permutation $v_1=\omega^{md}_{nd}$ has $d$ cycles of length $n$ and degree $m$, and in the second case $v_2=\omega^{md}_{nd} \sigma_1\cdots \sigma_{d-1}$ has one cycle of degree $md$. Finally, it is easy to see that $v_1$ and $v_2$ are minimal length representatives in their  respective conjugacy classes.
\end{proof}

For example, for $m=0$ and $n=1$ we get $\Tr(\Omega_d^0)=\Tr(\one_d)\simeq (H_{0,1})^{\star d}$ in agreement with \eqref{eq: power H zero one} and $\Tr(T_1\cdots T_{d-1})\simeq H_{0,d}$.

\begin{remark}
\label{rmk: parallel}
Topologically, the closures of affine braids correspond to link diagrams drawn on the torus, see \cite{GN Tr} for more context. In particular, $\Omega^{md}_{nd}$ corresponds to $d$ parallel copies of the $(m,n)$ torus knot while $\Omega^{md}_{nd}T_1\cdots T_{d-1}$ corresponds to a connected knot diagram where these $d$ copies are joined together with the minimal possible number $(d-1)$ of crossings.
\end{remark}

\begin{lemma}
\label{lem: centralizer pi m}
If $\gcd(m,n)=1$ then the centralizer of $\omega^{md}_{nd}$ in $\widetilde{S_{nd}}$ is  isomorphic to $\widetilde{S_d}$. For $d=1$ it is generated by $\omega_{nd}$, and for $d>1$ it is generated by $\omega_{nd}$ and $\widehat{\sigma_0},\widehat{\sigma_1},\cdots,\widehat{\sigma_{d-1}}$ where
$$
\widehat{\sigma_i}=\sigma_i\sigma_{i+d}\cdots \sigma_{i+nd-d}.
$$
\end{lemma}

\begin{proof}
Recall \cite{BB} that we can identify $\widetilde{S_{nd}}$ with the set of bijections $v:\Z\to \Z$ such that $v(i+nd)=v(i)+nd$. Under this identification, $\omega_{nd}$ corresponds to the shift $\omega_{nd}(i)=i+1$, so that $\omega_{nd}^{md}(i)=i+md$.

We have $v\omega^{md}_{nd}=\omega^{md}_{nd} v$ if $v(i+md)=v(i)+md$ for all $i$. Since $v(i+nd)=v(i)+nd,$  we conclude that $v(i+d)=i+d$ for all $i$, so $v$ corresponds to an element of $\widetilde{S_d}$.
The generators $\widehat{\sigma_i}$ of $\widetilde{S_d}$ swap $i+td$ with $(i+1)+td$ for all $t$, and indeed can be written in terms of generators of $\widetilde{S_{nd}}$ as described in the statement of Lemma \ref{lem: centralizer pi m}.
\end{proof}

The following result generalizes Theorem \ref{thm: vertical trace} to  the affine case and arbitrary slope $\frac mn$.

\begin{theorem}
\label{thm: endo omega m}
Suppose that  $\gcd(m,n)=1$. Then 
$$
\End_{\Tr(\ASBim_{nd})}\left(\Tr(\Omega^{md}_{nd})\right)\simeq \HH_*(R_d)\rtimes \widetilde{S_d}
$$
where $R_d=R/(x_i-x_{i+md})\simeq \C[x_1,\ldots,x_d]$. In particular,
$$
\End^0_{\Tr(\ASBim_{nd})}\left(\Tr(\Omega^{md}_{nd})\right)\simeq R_d\rtimes \widetilde{S_d}.
$$
\end{theorem}

\begin{proof}
We follow the proof of \cite[Theorem 7.6]{GHW}, and refer to \cite{GHW} for more details. By \eqref{eq: Hom dg trace}, the chain complex computing  
$\End_{\Tr(\ASBim_{nd})}\left(\Tr(\Omega^{md}_{nd})\right)
$
is generated by sequences of  morphisms
$$
\Hom_{\ASBim_{nd}}\left(\Omega^{md}_{nd} Z_k,Z_1\Omega^{md}_{nd}\right)\otimes \Hom_{\ASBim_{nd}} (Z_1,Z_2)\otimes \cdots \otimes \Hom _{\ASBim_{nd}}(Z_{k-1},Z_k).
$$
By \cite[Theorem 5.11]{GHW} we can use the semiorthogonal decomposition from Theorem \ref{thm:LW} to reduce to the case when $Z_i=T_{w_i}$ for some $w_i\in \widetilde{S_{nd}}$.  The spaces $\Hom(T_{w_i},T_{w_{i+1}})$ are nonzero only if $w_i\preceq w_{i+1}$, in particular $w_1\preceq w_k$ and $\ell(w_1)\le \ell(w_k)$. On the other hand,  
$$
\Hom_{\ASBim_{nd}}\left(\Omega^{md}_{nd} T_{w_k},T_{w_1}\Omega^{md}_{nd}\right)=\Hom_{\ASBim_{nd}}\left(\Omega^{md}_{nd} T_{w_k},\Omega^{md}_{nd} T_{w'_1}\right)=\Hom_{\ASBim_{nd}}(T_{w_k},T_{w'_1})
$$
where $w'_1=\omega_{nd}^{-md}w_1\omega^{md}_{nd}$. This is nonzero only if $w_k\preceq w'_1$, in particular $\ell(w_k)\le \ell(w'_1)=\ell(w_1)$. We conclude that $\ell(w_k)=\ell(w'_1)=\ell(w_1)$, and $w_1\preceq w_2\preceq \ldots\preceq w_k $ implies that $w_1=w_2=\ldots=w_k$. Furthermore, $w_k\preceq w'_1$ implies $w_k=w_1=w'_1$, so $w_1$ commutes with $\omega^{md}_{nd}$. 

We conclude that the complex for $\End_{\Tr(\ASBim_{nd})}(\Tr(\Omega^{md}_{nd}))$  is generated by the products
$$
\Hom_{\ASBim_{nd}}\left(\Omega^{md}_{nd}T_w,T_w\Omega^{md}_{nd}\right)\otimes \Hom_{\ASBim_{nd}}(T_w,T_w)\otimes \cdots \otimes \Hom_{\ASBim_{nd}}(T_w,T_w)\simeq 
$$
$$
R\otimes R\otimes \cdots \otimes R
$$
for $w$ in the centralizer $C(\omega^{md}_{nd})=\left\{w\omega^{md}_{nd}=\omega^{md}_{nd} w\right\}\simeq \widetilde{S_d}$. The last equation follows from Lemma \ref{lem: centralizer pi m}. The bar complex differential does not mix different $w$, but similarly to Lemma \ref{lem: endo trace rouquier} one can check   that it identifies $x_i$ with $x_{i+md}$ and sets $\delta=0$. If $R_d=R/(x_i-x_{i+md})$ then for each $w$ we get the bar complex computing 
$$
\HH_*(R_d)=\C[x_1,\ldots,x_{nd},\theta_1,\ldots,\theta_{nd}]/(x_i-x_{i+md},\theta_i-\theta_{i+md}).
$$
We denote the class of the canonical isomorphism in 
$\Hom_{\ASBim_{nd}}(\Omega^{md}_{nd} T_w,T_w\Omega^{md}_{nd})$ by $[w]$. Then the above computation shows that we have an isomorphism of vector spaces
$$
\End_{\Tr(\ASBim_{nd})}(\Tr(\Omega^{md}_{nd}))\simeq \bigoplus_{w\in C(\omega^{md}_{nd})} [w]\cdot \HH_*(R_d)\simeq  \HH_*(R_d)\rtimes \widetilde{S_d}.
$$
To complete the proof, we need to prove that the composition of various permutations $[w]\in \End_{\Tr(\ASBim_{nd})}(\Tr(\Omega^{md}_{nd}))$ satisfies the relations in the subgroup $C(\omega^{md}_{nd})$. For this, we argue as in \cite[Theorem 7.6]{GHW} and construct an action of $[w]\in \End_{\Tr(\ASBim_{nd})}(\Tr(\Omega^{md}_{nd}))$ on 
$$
\End_{\ASBim_{nd}}(\Omega^{md}_{nd})\simeq \widetilde{R_d}=\widetilde{R}/(x_i-x_{i+md})
$$
by the chain of isomorphisms
$$
\Hom_{\ASBim_{nd}}(\Omega^{md}_{nd},\Omega^{md}_{nd})\simeq \Hom_{\ASBim_{nd}}(\Omega^{md}_{nd} ,\Omega^{md}_{nd} T_wT_w^{-1})\simeq   \Hom_{\ASBim_{nd}}(\Omega^{md}_{nd} T_w,\Omega^{md}_{nd} T_w) \stackrel{\circ [w]}{\simeq}$$
$$\Hom_{\ASBim_{nd}}(T_w\Omega^{md}_{nd} ,\Omega^{md}_{nd} T_w)\simeq   \Hom_{\ASBim_{nd}}(\Omega^{md}_{nd} ,T_w^{-1}\Omega^{md}_{nd} T_w)\simeq  \Hom_{\ASBim_{nd}}(\Omega^{md}_{nd},\Omega^{md}_{nd}).
$$
One can check that this action is compatible with the composition of morphisms in the trace defined in \cite{GHW}, and induces the standard action of $C(\omega^{md}_{nd})$ on $\widetilde{R}_d$ up to a sign. Since this representation is faithful, we obtain the desired result.    
\end{proof}

In particular, we have an action of $S_d \subset \widetilde{S_d}$ on the object $\Tr\left(\Omega^{md}_{nd}\right)$ by endomorphisms. We are ready to define our main players. 

\begin{definition}

We define the {\bf slope $\frac mn$ Schur object} $\Tr\left(\Omega^{md}_{nd}\right)^{\lambda}\in \Tr(\ASBim_{nd})^{\Kar}$ by
$$
\Tr\left(\Omega^{md}_{nd}\right)^{\lambda}=\ee_{\lambda}\Tr\left(\Omega^{md}_{nd}\right).
$$
for any partition $\lambda \vdash d$. In particular, we define
$$
\bar{E}_{md,nd}:=\Tr\left(\Omega^{md}_{nd}\right)^{\mathrm{antisym}}.
$$
\end{definition}

For $m=0$ and $n=1$ we recover 
$
\bar{E}_{0,d}=\Tr\left(\one_d\right)^{\mathrm{antisym}}.
$

\begin{definition}
We define the {\bf slope $\frac mn$ subcategory}  $\Tr\left(\Omega^{md}_{nd}\right)^{\frac mn}$ as the smallest subcategory in $\Tr(\ASBim_{nd})^{\Kar}$ containing $\Tr\left(\Omega^{md}_{nd}\right)$ and closed under direct sums, direct summands, shifts and cones.
\end{definition}

\begin{lemma}
The slope subcategory $\Tr\left(\Omega^{md}_{nd}\right)^{\frac mn}$ is generated by the objects $\Tr\left(\Omega^{md}_{nd}\right)^{\lambda}$ under direct sums, shifts and cones. Furthermore,
$
K_0\left(\Tr\left(\Omega^{md}_{nd}\right)^{\frac mn}\right)
$ is freely generated over $\Z[q,q^{-1}]$ by $\Tr\left(\Omega^{md}_{nd}\right)^{\lambda}$, as $\lambda$ runs over partitions of $d$.
\end{lemma}

\begin{proof}
By Theorem \ref{thm: endo omega m} the endomorphism algebra of $\Tr\left(\Omega^{md}_{nd}\right)$ is non-negatively graded (with respect to the $q$-grading), with $\C[\widetilde{S_d}]$ in degree zero. 

Let $\mathcal{C}_{md,nd}$ be the smallest subcategory in $\Tr(\ASBim_{nd})^{\Kar}$ containing $\Tr\left(\Omega^{md}_{nd}\right)$ and closed under direct sums, direct summands and shifts (but not cones). If $M\in \mathcal{C}_{md,nd}$ is a direct sum of shifts of  $\Tr\left(\Omega^{md}_{nd}\right)$ then any idempotent endomorphism of $M$ has components in $\C[\widetilde{S_d}]$. Therefore the objects in $\mathcal{C}_{md,nd}$ are in bijection with the direct sums of shifts of projective modules over $\C[\widetilde{S_d}]$.

On the other hand, $\C[\widetilde{S_d}]=\C[S_d]\rtimes \C[y_1,\ldots,y_d]$. By a deep theorem of Swan \cite{Swan}, any projective module over $\C[y_1,\ldots,y_d]$ is free, hence any indecomposable projective module over $\C[\widetilde{S_d}]$ is isomorphic to $\ee_{\lambda}\C[\widetilde{S_d}]$ for some $\lambda\vdash d$. By translating back to $\Tr(\ASBim_{nd})^{\Kar}$, we conclude that any object of $\mathcal{C}_{md,nd}$ is isomorphic to a direct sum of $\Tr\left(\Omega^{md}_{nd}\right)^{\lambda}$.

Finally, the category $\mathcal{C}_{md,nd}$ is Karoubian, so its bounded homotopy category is Karoubian as well (see e.g. \cite[Theorem A.10]{GW}, \cite[Section 4]{GHW} and references therein). Therefore any direct summand of a (twisted) complex built from $\Tr\left(\Omega^{md}_{nd}\right)^{\lambda}$ is again a (twisted) complex built from $\Tr\left(\Omega^{md}_{nd}\right)^{\lambda}$.
\end{proof}

\subsection{Cabling functors}
\label{sec: cabling}

We would like to describe the action of $S_d$ on $\Tr(\Omega^{md}_{nd})$ in more concrete terms.

\begin{definition}
Fix $m,n$ and $d$ as above. Given an object $Z\in \SBim_d$, we define the object $\widehat{Z}\in \SBim_{nd}\subset \ASBim_{nd}$ as
$$
\widehat{Z}=\underbrace{Z\star Z\star Z\cdots \star Z}_{n \text{ times}}.
$$
\end{definition}
Recall that since $Z$ is an object of the finite Hecke category, $\widehat{Z}$ can be simply thought of a disjoint union of $n$ copies of the braid $Z$.

\begin{lemma}
\label{lem: Z hat}
For any $Z$ we have $\Omega^{md}_{nd}\widehat{Z}\simeq \widehat{Z}\Omega^{md}_{nd}$. 
\end{lemma}

\begin{proof}
If $d=1$, then $Z\in \SBim_1$ is the identity object, and so is $\widehat{Z}\in \SBim_n$. If $d>1$, it is sufficient to prove the statement for $Z=B_i,\ 1\le i\le d-1$. Note that $\Omega^m B_i=B_{i+m}\Omega^m$ and $\widehat{B_i}=B_iB_{i+d}\cdots B_{i+n-d}$ where all indices are understood modulo $n$.  Therefore
$$
\Omega^{md}_{nd} \widehat{B_i}\Omega^{-md}_{nd}=B_{i+md}B_{i+d+md}\cdots B_{i+nd-d+md}=B_iB_{i+d}\cdots B_{i+nd-d} = \widehat{B_i}.
$$
The middle equation follows from the fact that the factors in $\widehat{B_i}$ pairwise commute, and 
$$\{i,i+d,\ldots,i+nd-d\ \mathrm{mod}\ n\}=\{j\ \mathrm{mod}\ n: j=i\ \mathrm{mod}\ d\}=
$$
$$
\{i+md,i+d+md,\ldots,i+n-d+md\ \mathrm{mod}\ n\}.
$$
\end{proof}

Next, we define the {\em cabling functor} 
$$
\Phi^{\frac{m}{n}}:\Tr(\SBim_d)\to \Tr(\ASBim_{nd})
$$
as follows:

\begin{itemize}[leftmargin=*]

\item On the level of {\bf objects}, we have $\Phi^{\frac{m}{n}}(\Tr(X))=\Tr(\Omega^{md}_{nd} X)$
where we interpret $X$ as an object in $\SBim_d\subset \SBim_{nd}\subset \ASBim_{nd}$ supported on the first $d$ strands.

\item On the level of {\bf morphisms}, for any $Z\in \SBim_d$, $f\in \Hom_{\SBim_d}(XZ,ZY)$ we consider
\begin{equation}
\label{eq: f hat}
\widehat{f} : X\widehat{Z} = (XZ)\star Z\star Z\cdots \star Z \xrightarrow{f\star \Id\star \ldots\star \Id} (ZY)\star Z\star Z\cdots \star Z = \widehat{Z}Y
\end{equation}
and define $\Phi^{\frac{m}{n}}
(f)\in \Hom_{\ASBim_{nd}}\left(\Omega^{md}_{nd}X\widehat{Z},\widehat{Z}\Omega^{md}_{nd}Y\right)$ as the composition
\begin{equation}
\label{eq: Phi morphisms}
\Phi^{\frac{m}{n}}
(f): \Omega^{md}_{nd}X\widehat{Z}\xrightarrow{\text{Id} \otimes \widehat{f}}\Omega^{md}_{nd}\widehat{Z}Y\simeq \widehat{Z}\Omega^{md}_{nd}Y.
\end{equation}
(the $\simeq$ above uses Lemma \ref{lem: Z hat}). We can interpret $\Phi^{\frac{m}{n}}
(f)$ as a morphism in 
$$
\Hom_{\Tr(\ASBim_{nd})}\left(\Tr(\Omega^{md}_{nd}X),\Tr(\Omega^{md}_{nd}Y)\right)=\Hom_{\Tr(\ASBim_{nd})}\left(\Phi^{\frac{m}{n}}\Tr(X),\Phi^{\frac{m}{n}}\Tr(Y)\right).
$$

\item More generally, for a sequence of morphisms 
$$f_1\in \Hom_{\SBim_d}(XZ_k,Z_1Y),f_2\in \Hom_{\SBim_d}(Z_1,Z_2),\ldots,f_k\in \Hom_{\SBim_d}(Z_{k-1},Z_k)$$
representing by \eqref{eq: Hom dg trace} a class in the chain complex $\Hom_{\Tr(\SBim_d)}(X,Y)$ we define
$$\Phi^{\frac{m}{n}}(f_1,\ldots,f_k)=(g_1,\ldots,g_k)$$ where
$$
g_1:\Omega^{md}_{nd}X(Z_k\star Z_1\star\cdots\star Z_1)\xrightarrow{f_1\star \Id\star \cdots \star \Id}\Omega^{md}_{nd}\widehat{Z_1}Y\simeq \widehat{Z_1}\Omega^{md}_{nd}Y,
$$
and 
$$
g_i:Z_{i-1}\star Z_1\star \cdots \star Z_1\xrightarrow{f_i\star \Id\star \cdots \star \Id}Z_{i}\star Z_1\star \cdots \star Z_1
$$
for $i=2,\ldots,k$.
\end{itemize}

One can check that $\Phi^{\frac{m}{n}}$ defined above is a chain map on the respective $\Hom$ complexes and agrees with the composition of morphisms.

\begin{example}
By \eqref{eq: pmn from omega} we have
$$
\Phi^{\frac{m}{n}}(\Tr(T_1
\cdots T_{d-1}))=\Tr(\Omega^{md}_{nd}T_1
\cdots T_{d-1})\simeq H_{md,nd}.
$$
\end{example}

Note that in the above construction $X$ and $Z$ play very different roles. As in Remark \ref{rmk: parallel}, we can interpret $X\in \SBim_d$ topologically as a  pattern for cabling the $(m,n)$ torus knot. Given such a pattern on $d$ strands, we can always arrange for it to be supported on the first $d$ strands. In particular, we do not need to repeat $X$ or consider $\widehat{X}$.

In contrast, $\widehat{Z}$ represents morphisms in the trace which ``wrap around the annulus", see \cite{BHLZ,BPW,GHW,GW} for pictures (in particular, \cite[Section 3, Figures 4-5]{BPW}, \cite[Section 6.4,p.46]{GHW}) and more details. As such, it is naturally spread along $\Tr(\Omega^{md}_{nd})$. 

To phrase this differently, to swap two copies of the $(m,n)$ torus knot in $\Tr(\Omega^{md}_{nd})$, one needs to apply Reidemeister moves everywhere, and repeat them with period $d$, thus reproducing $\widehat{Z}$, as the following Lemma shows.

\begin{lemma}
\label{lem: slope schurs}
a) The functor $\Phi^{\frac{m}{n}}$ intertwines the actions of $S_d$ on $\Tr(\one_d)$ and $\Tr(\Omega^{md}_{nd})$.

b) For all $\lambda\vdash d$ we have 
$$
\Tr(\Omega^{md}_{nd})^{\lambda}=\Phi^{\frac{m}{n}}\left(\Tr(\one_d)^{\lambda}\right).$$ 
 In particular, we get
\begin{equation}
\Phi^{\frac{m}{n}}(\bar{E}_{d})=\bar{E}_{md,nd}.
\end{equation}
Here we extend $\Phi^{\frac{m}{n}}$ to the idempotent completions of the respective categories.
\end{lemma}

\begin{proof}
We have $\Phi^{\frac{m}{n}}(\Tr(\one_d))=\Tr(\Omega^{md}_{nd})$. Furthermore, the action of $S_d$ on $\Tr(\one_d)$ is induced by the identity maps in $\Hom_{\SBim_d}(\one_{d}T_i,T_i\one_d)$, $i=1,\ldots,d-1$.  These correspond to the identity maps in 
$\Hom_{\ASBim_{nd}}(\Omega^{md}_{nd}\widehat{T_i},\widehat{T_i}\Omega^{md}_{nd})$ which generate the action of $S_d$ on $\Tr(\Omega^{md}_{nd})$ by Theorem \ref{thm: endo omega m}, since the corresponding permutations $\widehat{\sigma_i}$ generate a copy of $S_d$ in the centralizer of $\omega^{md}_{nd}$ by Lemma \ref{lem: centralizer pi m}.

Therefore the $S_d$ actions on $\Tr(\one_d)$ and $\Tr(\Omega^{md}_{nd})$ agree under the functor $\Phi^{\frac{m}{n}}$, and
$$
\Phi^{\frac{m}{n}}\left(\Tr(\one_d)^{\lambda}\right)=\Phi^{\frac{m}{n}}(\ee_{\lambda}\Tr(\one_d))=\ee_{\lambda}\Phi^{\frac{m}{n}}(\Tr(\one_d))=\ee_{\lambda}\Tr(\Omega^{md}_{nd}) = \Tr(\Omega^{md}_{nd})^{\lambda}.
$$
\end{proof}

\begin{corollary}
\label{cor: schurs generate in slope 1}
For any $X\in \SBim_d$ the object $\Tr(\Omega^{md}_{nd}X)$ can be resolved by $\Tr(\Omega^{md}_{nd})^{\lambda}$. 
\end{corollary}

\begin{proof}
By Theorem \ref{thm: schurs generate finite} we can resolve $\Tr(X)$ by $\Tr(\one_d)^{\lambda}$. By Lemma \ref{lem: slope schurs} the functor $\Phi^{\frac mn}$ sends $\Tr(X)$ to $\Tr(\Omega^{md}_{nd}X)$ and $\Tr(\one_d)^{\lambda}$ to $\Tr(\Omega^{md}_{nd})^{\lambda}$, so the result follows.
\end{proof}

\begin{corollary}
\label{cor: schurs generate in slope 2}
Suppose that $\gcd(m,n) = 1$ and $d_1+\ldots+d_k=d$. Then the object $$H_{md_1,nd_1}\star\cdots\star H_{md_k,nd_k}\in \Tr(\ASBim_{nd})$$ can be resolved by $\Tr(\Omega^{md}_{nd})^{\lambda}$ or, equivalently, by the products
$$
\bar{E}_{md'_1,nd'_1}\star \cdots \star \bar{E}_{md'_s,nd'_s}
$$
with $d'_1+\ldots+d'_s=d$.
\end{corollary}

\begin{proof}
For $k=1$ we can use \eqref{eq: pmn from omega} to write $H_{m,n}=\Tr(\Omega^{md}_{nd}T_1\cdots T_{d-1})$ and apply Corollary \ref{cor: schurs generate in slope 1}. See also Theorem \ref{thm: ribbon Schurs slope} below for a more explicit computation.

For $k>1$  we can use \eqref{eq: omega from Pmn} and write
$$
\Tr(\Omega_{nd_1}^{md_1})\star\cdots \star\Tr(\Omega_{nd_k}^{md_k})\simeq \left(H_{m,n}\right)^{\star d_1}\star\cdots \star \left(H_{m,n}\right)^{\star d_k}\simeq 
\left(H_{m,n}\right)^{\star d}\simeq \Tr(\Omega^{md}_{nd}).
$$
The isomorphism is compatible with the action of $S_{d_1}\times \cdots\times S_{d_k}\subset S_{d_1+\ldots+d_k}$, hence any product
$$\Tr(\Omega_{nd_1}^{md_1})^{\lambda_1}\star \cdots \star\Tr(\Omega_{nd_k}^{md_k})^{\lambda_k}
$$
decomposes as a direct sum of $\Tr(\Omega^{md}_{nd})^{\lambda}$. The multiplicities are given by the coefficients of Schur functions $s_{\lambda}$ in the product $s_{\lambda_1}\cdots s_{\lambda_k}$.
By the $k=1$ case, we can resolve 
$H_{md_i,nd_i}$ by $\Tr(\Omega_{nd_i}^{md_i})^{\lambda_i}$, and hence we can resolve the product of $H_{md_i,nd_i}$ by $\Tr(\Omega^{md}_{nd})^{\lambda}$.
\end{proof}

In Section \ref{sec: ribbon Schur top} we provide an explicit computation of such a resolution for $k=1$. We also conjecture that this is a special case of a more general phenomenon.

\begin{conjecture}
The products
\begin{equation}
\label{eq: convex path schurs}
\bar{E}_{m_1d_1,n_1d_1}\star\cdots \star \bar{E}_{m_kd_k,n_kd_k}
\end{equation}
where $\sum n_id_i=n$ and $\frac{m_1}{n_1}\le \cdots \le \frac{m_k}{n_k}$ generate the idempotent completion $\Tr(\ASBim_n)^{\Kar}$. 
\end{conjecture}

Note that by Corollary \ref{cor: schurs generate in slope 2} and Theorem \ref{thm: convex paths} the objects \eqref{eq: convex path schurs} generate $\Tr(\ASBim_n)$. The problem is that we do not know if these objects are indecomposable for $k>1$, otherwise the idempotent completion might also contain their direct summands.

\subsection{Ribbon Schur functions and braids}
\label{sec: ribbon Schur top}

In this section we recall some results from \cite{GW}. 
To a sequence of signs $\varepsilon\in \{\pm \}^{n-1}$, we will associate several objects. First, we can consider the complex 
$$
C_{\varepsilon}=T_1^{\varepsilon_1}\cdots T_{n-1}^{\varepsilon_{n-1}}\in \SBim_n
$$
and its trace $\Tr(C_{\varepsilon})$. For example,
$$
\Tr(C_{+,\ldots,+})=\Tr(T_1\cdots T_{n-1})\simeq H_{0,n}.
$$

Second, we can consider the parabolic subgroups $W_{\varepsilon},W'_{\varepsilon}\subset S_n$ generated by simple reflections with positive (resp. negative) signs. Let $\ee_{\varepsilon}$ and $\ee^{-}_{\varepsilon}$ denote respectively the symmetrizer for $W_{\varepsilon}$ and the antisymmetrizer for $W'_{\varepsilon}$.

As in Section \ref{sub:ribbon} we define ribbon skew Schur functions $\bar{s}_{\varepsilon}$ associated to $\varepsilon$ (these are denoted by $\Psi(\varepsilon)$ in \cite{GW}).
For example, $\varepsilon=(-)^{n-1}$ corresponds to $\bar{s}_{(1^n)}=\bar{e}_n$. We also consider the plethystic substitution 
$
p_d=\bar{p}_d(1-q^d)
$
as in \eqref{eqn:plethysm}, and modified  ribbon skew Schur functions 
\begin{equation}
s_{\varepsilon}=\bar{s}_{\varepsilon}[\bar{p}_d\mapsto p_d].
\end{equation}

\begin{theorem}\cite{solomon}
\label{thm: solomon}
The group algebra $\C[S_n]$ can be presented as a direct sum of left ideals:
\begin{equation}
\label{eq: solomon}
\C[S_n] =\bigoplus_{\varepsilon\in \{\pm\}^{n-1}}\C[S_n]\ee_{\varepsilon}\ee^{-}_{\varepsilon}
\end{equation}
Furthermore, the character of the $S_n$--representation $V_{\epsilon}:=\C[S_n]\ee_{\varepsilon}\ee^{-}_{\varepsilon}$ equals the ribbon skew Schur function $\bar{s}_{\varepsilon}$.
\end{theorem}

\begin{example}
For $\varepsilon=(+)^{n-1}$ we have $W_{\varepsilon}=S_n$ and $W'_{\varepsilon}$ is trivial, so $V_{\varepsilon}$ is the trivial representation of $S_n$.
For $\varepsilon=(-)^{n-1}$ we have $W_{\varepsilon}$ is trivial and $W'_{\varepsilon}=S_n$, so $V_{\varepsilon}$ is the sign representation of $S_n$.
\end{example}

As a warning to the reader, for general $\varepsilon$ the representation $V_{\varepsilon}$ is not irreducible, see \cite[Example 2.18]{GW}.

We denote by $\pi_{\varepsilon}\in\C[S_n]$ the idempotent projecting to $V_{\varepsilon}$ in the decomposition \eqref{eq: solomon}.
Next, we consider the Koszul complex 
$$
K_n=(\Tr(\one)\otimes \wedge(\theta_1,\ldots,\theta_{n-1}),D)\in \Tr(\SBim_n),\quad D(\theta_i)=x_i-x_{i+1}.
$$
Recall that by Theorem \ref{thm: vertical trace} there is an action of $S_n$ on $\Tr(\one)$ by endomorphisms. This action extends to $K_n$ by declaring that $\theta_1,\ldots,\theta_{n-1}$ span a copy of the reflection representation. Then we have the following:

\begin{theorem}\cite[Theorem 5.25]{GW}
\label{thm: ribbon}
a) For all $\varepsilon$ we have in $\Tr(\SBim_n)^{\Kar}$:
$$
\Tr(C_{\varepsilon})[|\varepsilon|_{+}]=\pi_{\varepsilon}K_n.
$$
In particular,
$
\Tr(C_{+,\ldots,+})[n-1]=K_n^{\mathrm{sym}.}
$

b) The class of $\Tr(C_{\varepsilon})$ in $K_0\left(\Tr(\SBim_n)^{\Kar}\right)$ equals
$$
\left[\Tr(C_{\varepsilon})\right]=\frac{(-1)^{|\varepsilon|_{+}}}{1-q}s_{\varepsilon}.
$$
where $|\varepsilon|_{+}$ is the number of pluses in $\varepsilon$, and $q$ is the grading shift from equation \eqref{eqn:rouquier complexes}.

\end{theorem}

The product formula $s_{\varepsilon}s_{\varepsilon'}=s_{\varepsilon,+,\varepsilon'}+s_{\varepsilon,-,\varepsilon'}$ of \eqref{eqn:product formula} also has topological meaning. Let the lengths of $\varepsilon$ and $\varepsilon'$ 
be $n-1$ and $n'-1$ respectively, and consider the braids
$$
C_{\varepsilon}=T_1^{\varepsilon_1}\cdots T_{n-1}^{\varepsilon_{n-1}},\ 
C'_{\varepsilon'}=T_{n+1}^{\varepsilon'_1}\cdots T_{n+n'-1}^{\varepsilon_{n'-1}}.
$$ 
There is an exact triangle relating 
$C_{\varepsilon}C'_{\varepsilon'}=C_{\varepsilon}\star C_{\varepsilon'},$
$C_{\varepsilon,+,\varepsilon'}=C_{\varepsilon}T_nC'_{\varepsilon'}$ and 
$C_{\varepsilon,-,\varepsilon'}=C_{\varepsilon}T_n^{-1}C'_{\varepsilon'}$
which actually splits in the trace \cite[Theorem 5.25]{GW}, so that
$$
\left[q\Tr(C_{\varepsilon})\star \Tr(C_{\varepsilon'})\to \Tr(C_{\varepsilon})\star \Tr(C_{\varepsilon'})\right]\simeq \Tr(C_{\varepsilon,+,\varepsilon'})[1]\oplus \Tr(C_{\varepsilon,-,\varepsilon'}).
$$
On the level of Grothendieck groups one gets
$$
(1-q)[\Tr(C_{\varepsilon})]\cdot [ \Tr(C_{\varepsilon'})]=-[\Tr(C_{\varepsilon,+,\varepsilon'})]+[\Tr(C_{\varepsilon,-,\varepsilon'})],
$$
as expected by Theorem \ref{thm: ribbon} b).

\begin{example}
For $n=2$ we have $\Tr(\one)\simeq \Tr(\one)^{\mathrm{sym}}\oplus \Tr(\one)^{\mathrm{antisym}}$. Then
$$
K_2=\left[q\Tr(\one)\xrightarrow{x_1-x_2}\Tr(\one)\right]
$$
decomposes into
$$
(K_2)^{\mathrm{sym}}=\left[q\Tr(\one)^{\mathrm{antisym}}\xrightarrow{x_1-x_2} \Tr(\one)^{\mathrm{sym}}\right]
$$
and
$$
(K_2)^{\mathrm{antisym}}=\left[q\Tr(\one)^{\mathrm{sym}}\xrightarrow{x_1-x_2} \Tr(\one)^{\mathrm{antisym}}\right].
$$
Note that the actions of $S_2$ on the two copies of $\Tr(\one)$ in $K_2$ differ by a sign, since $x_1-x_2$ is an antisymmetric polynomial.
Furthermore, $$\overline{s}_{+}=\frac{\bar{p}_1^2+\bar{p}_2}{2},\ \overline{s}_{-}=\frac{\bar{p}_1^2-\bar{p}_2}{2}$$
so
$$
\frac{1}{1-q}s_{+}=\frac{(1-q)\bar{p}_1^2+(1+q)\bar{p}_2}{2}=\bar{h}_{2}-q\bar{e}_{2},
$$
while
$$
\frac{1}{1-q}s_{-}=\frac{(1-q)\bar{p}_1^2-(1+q)\bar{p}_2}{2}=\bar{e}_{2}-q\bar{h}_{2}.
$$
\end{example}

Finally, we can apply the functor $\Phi^{\frac{m}{n}}$ to transport these results to an arbitrary slope.

\begin{theorem}
\label{thm: ribbon Schurs slope}
Let $\gcd(m,n)=1$ and $d\ge 1$.
Define the Koszul complex
$$
K_{md,nd}=\left(\Tr(\Omega^{md}_{nd})\otimes \wedge(\theta_1,\ldots,\theta_{d-1}),D\right)\in \Tr(\ASBim_{nd}),\quad D(\theta_i)=x_i-x_{i+1}.
$$
There is a natural action of $S_d$ on $K_{md,nd}$ and for all $\varepsilon\in \{\pm\}^{d-1}$ one has
$$
\Tr(\Omega^{md}_{nd}C_{\varepsilon})\simeq \pi_{\varepsilon}K_{md,nd}.
$$
In particular,
$$
H_{md,nd}\simeq \Tr(\Omega^{md}T_1\cdots T_{d-1})\simeq \Tr(\Omega^{md}C_{+,\ldots,+})=K_{md,nd}^{\mathrm{sym}}.
$$
\end{theorem}

\begin{proof}
We use the functor $\Phi^{\frac{m}{n}}$ which sends $\Tr(\one_d)$ to $\Tr(\Omega^{md}_{nd})$ and $\Tr(C_{\varepsilon})$ to $\Tr(\Omega^{md}_{nd}C_{\varepsilon})$. Clearly, we have 
$$
\Phi^{\frac{m}{n}}(K_d)=K_{md,nd},
$$
and by Lemma \ref{lem: slope schurs} $\Phi^{\frac{m}{n}}$ intertwines the actions of $S_d$ on $\Tr(\one_d)$ and $\Tr(\Omega^{md}_{nd})$, and on $K_d$ and $K_{md,nd}$ respectively, so the result follows.
\end{proof}

\begin{remark}
\label{rem: H schurs slope}
The complex $K_{md,nd}^{\mathrm{sym}}$ has terms 
$$\left(\Tr(\Omega_{nd}^{md})\otimes \wedge^i(\theta_1,\ldots,\theta_{d-1})\right)^{\mathrm{sym}}\simeq \Tr(\Omega_{nd}^{md})^{\lambda_i}
$$
where $\lambda_i=(n-i,1^i)$ are the hook partitions. Thus Theorem \ref{thm: ribbon Schurs slope} provides an explicit resolution of $H_{md,nd}$ by slope $\frac{m}{n}$ Schur objects, which can be further resolved by the products of $\overline{E}_{ma,na}$.
\end{remark}

\end{document}